\documentclass[11pt,reqno]{amsart}

\usepackage[T1]{fontenc}
\usepackage{lmodern}
\usepackage{amsmath,amssymb,amsthm,mathtools,todonotes}
\usepackage{microtype}
\usepackage{booktabs,array}
\usepackage{tikz}
\usetikzlibrary{arrows.meta,positioning,calc,fit,backgrounds}
\usepackage[hmargin=1.1in,vmargin=1in]{geometry}
\usepackage{placeins}
\usepackage[colorlinks=true,linkcolor=blue!45!black,citecolor=blue!45!black,
  urlcolor=blue!45!black]{hyperref}
\hypersetup{pdftitle={Thompson's groups are sofic},
  pdfauthor={Maksym Chaudkhari, Kate Juschenko, Kyrylo Muliarchyk}}

\newtheorem{theorem}{Theorem}[section]
\newtheorem{proposition}[theorem]{Proposition}
\newtheorem{lemma}[theorem]{Lemma}
\newtheorem{claim}[theorem]{Claim}
\newtheorem{corollary}[theorem]{Corollary}
\theoremstyle{definition}
\newtheorem{definition}[theorem]{Definition}

\newtheorem{remark}[theorem]{Remark}
\theoremstyle{remark}
\numberwithin{equation}{section}
\DeclareMathOperator{\Sym}{Sym}
\newcommand{\id}{\mathrm{id}}
\newcommand{\Prob}{\mathbb P}
\newcommand{\FF}{\mathbb F}
\newcommand{\eps}{\varepsilon}
\newcommand{\row}[2]{#1\!:\!(#2)}
\newcommand{\dH}{d_{\mathrm H}}
\newcommand{\len}{\ell}
\newcommand{\emptyword}{\varnothing}

\title[Thompson's groups are sofic]{Thompson's groups are sofic}
\author{Maksym Chaudkhari}
\address{M.C., Department of Mathematics and Statistics, University of South Florida, 4202 E Fowler Ave, CMC 342, Tampa, FL 33620, USA}
\email{mc637@usf.edu, maksymchaudkhari@utexas.edu}
\author{Kate Juschenko}
\address{K.J., Department of Mathematics, University of Texas at Austin, 2515 Speedway C1200, Austin, TX 78712, USA}
\email{katie.juschenko@gmail.com}
\author{Kyrylo Muliarchyk}
\address{K.M., Mathematics Department, Bar-Ilan University, Ramat Gan, 5290002, Israel}
\email{kyrylo.muliarchyk@biu.ac.il}
\date{\today}
\subjclass[2020]{Primary 20F65; Secondary 20F69, 20B30}
\keywords{Thompson's group, sofic group, hyperlinear group, sofic approximation,
stability}

\begin{document}
\maketitle
\begin{abstract}
We prove that Thompson's groups $F$, $T$ and $V$ are sofic. As a corollary, all of them are also hyperlinear and not permutation stable. The proof relies on an explicit construction of sofic approximations to $F$ and $V$. Moreover, the construction allows us to obtain effective quantitative bounds on the parameters of sofic approximations.
\end{abstract}
\maketitle

\section{Introduction}

Defined by Richard Thompson in 1965, Thompson's groups $F$, $T$, and $V$ are a classical family of groups of dynamical origin. These groups are infinite and finitely presented, and the reader can find their standard presentations, for example in \cite{CFP}. The standard presentation of $F$ has two generators and two relators given by commutators. Moreover, a presentation of $V$ that uses just two generators is described in \cite[Theorem 1.3]{BQ}.

Thompson's groups $F$, $T$, and $V$ have natural dynamical interpretations. Thompson's group $F$ is defined as a group of increasing piecewise linear homeomorphisms of $[0,1]$ with finitely many dyadic breakpoints and slopes that are integral powers of $2$. Furthermore, all of Thompson's groups can be defined in terms of their action on the Cantor space by prefix replacement transformations.

Thompson's groups are known for their rigidity properties. For example, the groups $T$, $V$ and the commutator subgroup of $F$ are simple. In particular, these groups are not residually finite, and since $F$, $T$, and $V$ are finitely presented, Thompson's groups are also not locally embeddable into finite groups (LEF). Furthermore, the questions about their approximation properties turned out to be notoriously difficult. For example, while non-amenability of Thompson's groups $T$ and $V$ was established quite early, the non-amenability question for $F$ became a major open problem in geometric group theory that was resolved only in September 2026 by an AI model developed by OpenAI, see \cite{OAI_F_26}, or \cite{KL26} for exposition. Soficity and hyperlinearity of Thompson's groups $F$, $T$, and $V$ were also longstanding open questions. We will recall the definition of soficity in Section~\ref{sec:prelim}, and  informally, a discrete group is sofic if the multiplication tables of its finite subsets can be simulated by permutations of a finite set up to a small error in normalized Hamming distance. The class of all sofic groups includes all amenable, all residually finite groups and all LEF groups. Moreover, finding an example of a non-sofic group was a major open problem with the first example constructed by an AI model developed by OpenAI in \cite{OAI_Sof_26}. Thompson's groups $F$, $T$ and $V$ were also considered as potential examples of non-sofic groups. The main goal of this article is to prove the following theorem.

\begin{theorem}\label{thm:V-intro}
Thompson's group $V$ is sofic. 
\end{theorem}

As a corollary, we obtain that the groups $F$ and $T$ are sofic, as well.

\begin{corollary}\label{cor:TFV}
Thompson's groups $F$, $T$, and $V$ are sofic, and hence hyperlinear. Moreover, they are not permutation-stable.  
\end{corollary}

We refer the reader to \cite[Section 3]{AP} for the definition of stability in permutations.
\begin{proof}
The first claim follows immediately from the fact that soficity passes to subgroups, and $F$ and $T$ can be realized as subgroups of $V$. Then hyperlinearity of Thompson's groups is an immediate corollary of soficity, see \cite[Theorem 2]{ES05}.
The second claim follows from the fact that every sofic permutation-stable group must be residually finite, see for example \cite[Theorem 7.2]{AP}, but none of Thompson's groups is residually finite.
\end{proof}

The proof of Theorem ~\ref{thm:V-intro} relies on an explicit construction of a sequence of sofic approximations for $V$. Since this construction is easier to describe in the case of Thompson's group $F$, we will postpone the proof of Theorem \ref{thm:V-intro} until the last section of the article, and we will first give a detailed proof of the following theorem.
\begin{theorem}\label{thm:main}

Thompson's group $F$is sofic.
\end{theorem}

\subsection{Outline of the argument} Our proofs of Theorem ~\ref{thm:V-intro} and Theorem ~\ref{thm:main} rely on an explicit construction of a sequence of sofic approximations for $V$ and $F$, respectively.

In order to prove Theorem~\ref{thm:main}, we construct a sequence of  finite sets $\{\Omega_n\}_{n \in N}$ and permutations $X_n, Y_n \in \Sym (\Omega_n)$ that will simulate the standard generators $x$ and $y$ of Thompson's group $F$. 
The finite sets $\Omega_n$ are explicitly defined as follows
\[
  \Omega_n=\{(L,R):L,R:\{0,\ldots,n-1\}\longrightarrow
                         \{0,\ldots,n-1\}\}.
\]
An element of \(\Omega_n\) can also be viewed as a rooted directed labeled graph with the vertex set $\{0, \ldots, n-1\}$ and a root at $0$, where each vertex has exactly one outgoing edge labeled $L$ and one outgoing edge labeled $R$.

The permutations $X_n$ and $Y_n$ that approximate generators $x$ and $y$ of $F$ act on the set of these graphs by changing outgoing edges at the root $0$ and at most two other vertices close to the root. In the generic case, the diagrams describing these changes are provided in Figures \ref{fig:x}--\ref{fig:y}. The reader may notice that they closely resemble the tree pair diagrams defining elements of $F$. This analogy becomes more pronounced, when we introduce the notion of an \textit{address} of a number from $\{0, \ldots, n-1\}$ with respect to a given element of $\Omega_n$. The address of a number $m$ will be essentially an encoding of the shortest directed path from $0$ to $m$ by a word over the binary alphabet $\{L,R\}$, when such a path exists. The maps $X_n$ and $Y_n$ will act on most addresses by the same prefix-replacement transformations that define the standard action of $F$ on the Cantor set. However, there is an important caveat. The  standard prefix-replacement action of $F$ on the Cantor space induces only a partial action of $F$ on the set of all finite binary words $W=\{0,1\}^*$. In our case, the addresses are finite words, so we will first define an action $\sigma$ of $F$ on $W$ that fixes the empty word and extends the partial action of $F$. In the generic case, this action will precisely describe the changes of the addresses of \textit{all} numbers close to $0$ in the rooted graph. Using this action, we will show that any fixed word $\rho$ in the free group $\FF(x,y)$, when evaluated over $X_n$ and $Y_n$ for some large enough $n$, will fix a generic point of $\Omega_n$ if and only if $\rho$ evaluates to a trivial element of $F$ as a word over the standard generators $x$ and $y$. However, if $\rho$ defines a non-trivial element in $F$, the corresponding permutation in $\Sym(\Omega_n)$ will move most of the points of $\Omega_n$.

\subsection{Organization of the article}
The first part of our article will mostly focus on Thompson's group $F$.
Section~\ref{sec:prelim} contains the necessary background material on sofic groups and Thompson's groups $F$ and $V$.

In Section~\ref{sec:fw_action} we construct the action $\sigma$ of Thompson's group $F$ on the set of all finite binary words $W$ that will extend its standard partial action on $W$ while also fixing the empty word.

Section~\ref{sec:maps_and_approximations} contains the explicit description of sofic approximations to $F$. In particular, this is the section where introduce the sets $\Omega_n$ and study the action of words over $X_n$ and $Y_n$ on the points from special subsets $G_n(r)$ of $\Omega_n$.

In Section~\ref{sec:finishing_the_proof}, we show that for a large enough $n$, a special set $G_n(r)$ contains almost all of the points from $\Omega_n$, and then present a proof of an even stronger quantitative version of Theorem~\ref{thm:main}, see Theorem~\ref{thm:F_sofic_quantitative}.

Finally, in Section~\ref{sec:final_proof_for_V}, we explain how one can adapt the arguments from Sections~\ref{sec:fw_action} --\ref{sec:finishing_the_proof} to establish an even stronger quantitative version of Theorem~\ref{thm:V-intro}, see Theorem~\ref{thm:V-quantitative-sofic}. Moreover, we list several classes of groups that embed into $V$, and hence are also sofic.

\subsection{Acknowledgments.} The authors are grateful to Dmytro Savchuk for helpful conversations on the topic and for the comments on the first draft of the article.

\subsection{AI use statement.}
The first version of the proof of soficity for Thompson's group $F$ was produced by
ChatGPT 6 Astra in a long session with the first author on October 6, 2026. The argument used the sets $\Omega_n$ but was weaker than the one presented in the article. It relied on explicit  evaluations of the commutator of $X_n$ and $Y_n$ and the two relators of $F$ expressed as words over $X_n$ and $Y_n$ to produce an embedding of $F$ in a metric ultraproduct of finite symmetric groups, followed by an application of \cite[Theorem 1]{ES05}. In particular, it did not allow for effective control over the sofic approximations for $F$. The proof presented in Sections \ref{sec:fw_action}--\ref{sec:final_proof_for_V} was developed by the authors. ChatGPT 6 Astra was also used for assistance with illustrations and proofreading. The authors take full responsibility for the contents of the article.

\section{Preliminaries}\label{sec:prelim}

\subsection{Notation}
We consider left group actions, so the permutations are composed from right to left. Our commutator convention is
\[
 [u,v]=uvu^{-1}v^{-1}.
\]

Furthermore, for $n\geq1$, we use the notation $[n]=\{0,\ldots,n-1\}$. For the semigroup of functions from $[n]$ to $[n] $ we will consider right composition to ensure compatibility with the concatenation of words in the free monoids.

Finally, throughout this article, we denote by $\mathcal{C}$ the Cantor space realized as the space of infinite binary sequences, $\mathcal{C}=\{0,1\}^{\mathbb{N}}$ with the product topology. The set of all finite binary words is denoted by $W=\{0,1\}^{*}$. Equipped with the operation of concatenation of words, $W$ is a free monoid generated by $\{0,1\}$.
For every $\alpha \in W$, we denote by $\alpha\mathcal{C}$ the clopen subset of $\mathcal{C}$ that consists of all $w \in \mathcal{C}$ with prefix $\alpha$.

\subsection{Sofic groups}
We recall the definition of sofic groups below, and we refer the reader to \cite{P08} for a more detailed introduction to sofic and hyperlinear groups.

For a nonempty finite set $A$, write $\Sym(A)$ for its permutation
group and put
\begin{equation}\label{eq:hamming}
 \dH(s,t)=\frac{|\{a\in A:s(a)\ne t(a)\}|}{|A|},
 \qquad \len(s)=\dH(s,\id).
\end{equation}
It is easy to see that the metric $\dH$ is invariant under both left and right
multiplication. 

\begin{definition}\label{def:sofic}
A group $G$ is \emph{sofic} if, for every finite set $E\subseteq G$
containing $id_G$ and every $\eps \in (0,1)$, there are a nonempty finite set
$A$ and a map $\varphi:E\to\Sym(A)$, with $\varphi(id_G)=\id$, such that
\begin{align*}
&\forall g,h,gh\in E: \; \dH(\varphi(g)\varphi(h),\varphi(gh)) \leq \eps.&
       \\
 &\forall g\in E\setminus\{id_G\}:   \; \dH(\varphi(g),\id) \geq 1-\eps. 
\end{align*}
For fixed $E$ and $\epsilon$, we will also call any such pair an $(E, \epsilon)$-sofic approximation to $G$.
\end{definition}

Moreover, for a finitely generated group $G$ with a chosen symmetric finite generating set $S$, soficity of $G$ admits an equivalent geometric definition.

\begin{definition}[Gromov's condition]
A finitely generated group $G$ is sofic if for any finite symmetric generating set $S$ of $G$, any $R >0$, and any $\epsilon \in (0,1)$, there exists a finite marked graph $\Gamma$ with edges labeled by $S$, such that the following condition holds. If we pick a vertex $v$ of $\Gamma$ from the uniform distribution, then with probability at least $(1-\epsilon)$, the ball of radius $R$ centered at $v$ in $\Gamma$ is isomorphic, as a rooted labeled graph, to the ball of radius $R$ in the Cayley graph of $G$ with respect to $S$.
\end{definition}

When the generating set $S$ is fixed and clear from the context, we will call such a graph $\Gamma$ an $(R,\epsilon)$-marked approximation to $G$.

\subsection{Thompson's groups $F$ and $V$}

We briefly recall the definitions and basic properties of Thompson's groups $F$ and $V$, and we refer the reader to \cite{CFP} for a more detailed introduction to Thompson's groups, and in particular, for their definitions in terms of pairs of finite binary trees and permutations. 
\begin{definition}
Thompson's group $F$ is the group of all orientation-preserving piecewise linear homeomorphisms $g$ of the unit segment $[0,1]$ that satisfy the following two conditions.

\begin{enumerate}
    \item The homeomorphism $g$ has finitely many breakpoints and all of them are dyadic rationals.
    \item All slopes of $g$ are integer powers of $2$.
\end{enumerate}

\end{definition}

Thompson's group $F$ is finitely generated, and its standard generating set is given by the two maps defined below.

\begin{equation*}\label{eq:F_PL_gens_prelim}
 x(t)=
 \begin{cases}
  t/2,&0\leq t\leq1/2,\\
  t-1/4,&1/2\leq t\leq3/4,\\
  2t-1,&3/4\leq t\leq1,
 \end{cases}
 \qquad
 y(t)=
 \begin{cases}
  t,&0\leq t\leq1/2,\\
  \dfrac{1+x(2t-1)}2,&1/2\leq t\leq1.
 \end{cases}
\end{equation*}

Moreover, equipped with its standard finite generating set, $F$ has the following finite presentation
\begin{equation}\label{eq:presentation}
 F=\langle x,y\mid \rho_1=1,\ \rho_2=1\rangle, \text{ where }
 \rho_1=[xy^{-1},x^{-1}yx] \text{ and }
 \rho_2=[xy^{-1},x^{-2}yx^2]. 
\end{equation}

We will denote the corresponding projection from the free group on two generators $\FF$ onto $F$ by $\pi: \FF \to F$.

The commutator subgroup $F'$ of $F$ is simple, and the only proper quotients of $F$ are abelian groups, see for example \cite[Theorem~4.3]{CFP}. Thus, it is not residually finite.

We are going to state the definition of Thompson's group $V$ in terms of prefix-replacement transformations. Recall that a collection of finite binary words $A=\{\alpha_1, \ldots, \alpha_k\} \subset W$ is called a \textbf{complete prefix code} if no word in $A$ is a proper prefix of a word from $A$, and every infinite word $\omega \in \mathcal{C}$ has a prefix in $A$. We will also call an ordered list of finite binary words $(\alpha_1, \ldots, \alpha_k)$ a complete prefix code if its entries are pairwise distinct and the set $\{\alpha_1, \ldots, \alpha_k\}$ is a complete prefix code.

\begin{definition}
A homeomorphism $\varphi: \mathcal{C} \to \mathcal{C}$ is called a prefix-replacement transformation if there exist two complete prefix codes of the same length $(\alpha_1, \ldots, \alpha_k)$ and $(\beta_1, \ldots, \beta_k)$ such that
for any $w \in \mathcal{C}$ and every $i=1, \ldots, k$, we have
\[\varphi(\alpha_iw)=\beta_iw.\]
In this case, we will also say that $\varphi$ replaces $(\alpha_1, \ldots, \alpha_k)$ by $(\beta_1, \ldots, \beta_k)$.
\end{definition}

It is easy to see that any pair of two complete prefix codes of the same length defines a prefix-replacement transformation. However, one prefix-replacement transformation may be defined by multiple different pairs of complete prefix codes.  We will call a pair of complete prefix codes $A=(\alpha_1, \ldots, \alpha_k)$ and $B=(\beta_1, \ldots, \beta_k)$ reduced, if there are no words $\alpha$ and $\beta$ such that  for some $i,j$ we have $\alpha_i=\alpha0$ and $\beta_i=\beta_0$ together with $\alpha_j=\alpha1$ and $\beta_j=\beta1$.

\begin{definition}
Let $\varphi$ be a prefix-replacement transformation of $\mathcal{C}$. We will call a pair $(\alpha_1, \ldots, \alpha_k)$ and $(\beta_1, \ldots, \beta_k)$ the prefix-replacement representation of $\varphi$ if this pair defines $\varphi$ and is reduced.
\end{definition}

It is easy to see that the prefix-replacement representation of $\varphi$ is unique up to permutation of the indices. From now on, we always work only with reduced pairs of complete prefix codes.

\begin{definition}
Thompson's group $V$ is defined as the group of all prefix-replacement transformations of $\mathcal{C}$.
\end{definition}

In this interpretation, Thompson's group $F$ is the subgroup of all prefix-replacement transformations of $\mathcal{C}$ that preserve the lexicographic order on $\mathcal{C}$. In particular, the standard generator $x$ replaces the complete prefix code $(0,10,11)$ by $(00,01,1)$, and $y$ replaces $(0,10,110,111)$ by $(0,100,101,11)$.

For every word $\alpha \in W$, we denote by $\mathrm{Rist}_V(\alpha\mathcal{C})$ the subgroup of all elements of $V$ with the support in $\alpha\mathcal{C}$. This subgroup is called \textit{the rigid stabilizer} of $\alpha\mathcal{C}$. It is easy to see that $\mathrm{Rist}_V(\alpha\mathcal{C})$ is isomorphic to $V$. For example, for the homeomorphism $\hat{\alpha}: \mathcal{C} \to \alpha\mathcal{C}$ given by $\omega  \mapsto \alpha\omega$, we have $V=\hat{\alpha}^{-1}\mathrm{Rist}_V(\alpha\mathcal{C})\hat{\alpha}$.

Finally, we note that the prefix-replacement action of $V$ on $\mathcal{C}$ induces a partial action of $V$, and consequently of $F$, on the set $W$. More explicitly, if $g \in V$ has a prefix-replacement representation $(\alpha_1, \ldots, \alpha_k) \to(\beta_1, \ldots, \beta_k)$ and $w \in W$, then the partial bijection $P_g$ is not defined on $w$ if it is a proper prefix of some word $\alpha_i \in (\alpha_1, \ldots, \alpha_k)$, and otherwise $w$ has the form $w=\alpha_iw'$, where $\alpha_i$ is the unique prefix of $w$ in the list $(\alpha_1, \ldots, \alpha_k)$. In this case, we set $P_g(w)=\beta_iw'$.

\section{An action of $F$ on finite binary words}\label{sec:fw_action}
Recall that $W=\{0,1\}^*$ denotes the free monoid on two letters which we will also view as the set of all finite binary words. As we explained above, the standard action of Thompson's group $F$ on the Cantor set $\mathcal{C}$ by prefix replacement transformations induces a partial action of $F$ on $W$. For $g \in F$, we denote a corresponding partial bijection by $P_g: W \rightarrow W$.
In this section we construct a faithful action $\sigma$ of $F$ on $W$ that fixes $\emptyword$ and extends the partial action of $F$ on $W$. This means that for any $g \in F$ and any word $w \in W$ such that $P_g(w)$ is defined, we have 
\[P_g(w)=\sigma_g(w).\]

Set $\mathcal D=\mathbb Z[1/2]\cap(0,1]$, and for every word $w=w_1\cdots w_k \in W$ define
\[0.w1=\sum_{j=1}^k w_j2^{-j}+2^{-(k+1)}.\] Thus, the number $0.w1$ is the midpoint of the standard dyadic interval with the address $w$, and $\emptyword$ is sent to $0.1=1/2$. The midpoint map $m:w \mapsto 0.w1$ is the standard embedding of $W$ into $\mathbb Z[1/2]\cap(0,1)$, with the following property. If $g \in F$ replaces a prefix $\alpha$ by a prefix $\beta$ in the action of $F$ on $\mathcal{C}$, then for any binary word $u \in W$ we have: 
\[g(m(\alpha u))=g(0.\alpha u1)=0.\beta u1=m(P_g(\alpha u))\]
Next, we will slightly modify the midpoint map\footnote{Even though for every $g \in F$, $m^{-1}gm$ is a permutation of $W$ that extends $P_g$, these permutations do not fix $\emptyword$, since $m(\emptyword)=1/2$ is not a fixed point of $F$ in $(0,1]$, so we need to make a small change.} to obtain a bijection $\kappa: W \longrightarrow D$ such that for each $g \in F$ the map 
$\sigma_g=\kappa^{-1}g\kappa$ extends $P_g$ to a permutation of $W$ that fixes $\emptyword$. We define $\kappa$  explicitly by the formula
\begin{equation}\label{eq:F-kappa}
 \kappa(w)=\begin{cases}
  1,&w=\emptyword,\\
  1-2^{-k}=0.w,&w=1^k,\quad k\geq1,\\
  0.w1,&w\text{ contains a }0.
 \end{cases}
\end{equation}

It is easy to see that $\kappa$ is also a bijection and the empty word $\emptyword$ is mapped to $1$ - the fixed point of the action of $F$ on $(0,1]$.

\begin{proposition}\label{prop:F-word-action}
 The map $\sigma : F \to \Sym (W)$
\begin{equation}\label{eq:F-sigma}
  g \mapsto \sigma_g=\kappa^{-1}g\kappa\qquad 
\end{equation}
 defines a faithful action of $F$ on $W$ that fixes $\emptyword$. Moreover, for any $g \in F$ and any word $w \in W$ such that $P_g(w)$ is well-defined, we have \[P_g(w)=\sigma_g(w).\]
\end{proposition}

\begin{proof}

Since no non-trivial element of $F$ fixes the entire set $\mathcal{D}$, the action defined by $\sigma$ is faithful and it fixes $\emptyword$ by the definition of $\kappa$.

Next, we check that for every $g \in F$ and every $w$ such that $P_g(w)$ is well defined, we have \[\sigma_g(w)=P_g(w).\]
Suppose that in the prefix replacement representation of $F$, $g \in F$ replaces the complete prefix code $(\alpha_1 ,\ldots, \alpha_l)$ by the complete prefix code $(\beta_1, \ldots, \beta_l)$ while preserving the lexicographic order. Notice that the only prefixes in each code that do not contain $0$ are $\alpha_l$ and $\beta_l$.
If $\alpha_l=1^m$  and $\beta_l =1^h$ for some $m, h \geq 1$, then $P_g$ is not defined for any word $w$ of the form $w=1^n, \, 1 \leq n<m$. Now we consider two cases.

If $w=1^n$ for some $n \geq m$, then we have $g(0.1^n)=0.1^{n+h-m}$ in $\mathcal{D}$ and $g(1^n)=1^{n+h-m}$ in $W$, so the equality $\sigma_g(w)=P_g(w)$
holds in this case.

Next, consider the case when $w$ contains the symbol $0$ and $P_g(w)$ is well-defined. We claim that $P_g(w)$ contains $0$ as well. Indeed, if $\alpha_l=1^m$ is a prefix of $w$, then $0$ occurs in the position with index $j >m$, so replacing the prefix $\alpha_l=1^m$ of $w$ by $\beta_l=1^h$ does not affect this $0$. If $\alpha_l$ is not a prefix of $w$, then the action of $P_g$ on $w$ replaces the prefix $\alpha_i$ of $w$ by a prefix $\beta_i$ for some $1 \leq i<l$. In this case, $0$ appears in $\beta_i$, so $P_g(w)$ contains $0$ as well.

 Finally, since $w$ and $P_g(w)$ both contain $0$, we have equalities $\kappa(w)=m(w)$ and $\kappa(P_g(w))=m(P_g(w))$, so the
 equality $\sigma_g(w)=P_g(w)$ follows from the equality $m^{-1}gm(w)=P_g(w)$ that holds by the definition of $m$.
 \end{proof}

Next, we write down the explicit formulas for $\sigma_x$ and $\sigma_y$, where $x$ and $y$  are the standard generators of $F$. Recall that $x$ and $y$ are  explicitly defined by the formulas
\begin{equation}\label{eq:F-PL-generators}
 x(t)=
 \begin{cases}
  t/2,&0\leq t\leq1/2,\\
  t-1/4,&1/2\leq t\leq3/4,\\
  2t-1,&3/4\leq t\leq1,
 \end{cases}
 \qquad
 y(t)=
 \begin{cases}
  t,&0\leq t\leq1/2,\\
  \dfrac{1+x(2t-1)}2,&1/2\leq t\leq1.
 \end{cases}
\end{equation}
Then, a direct computation shows that $\sigma_x$ and $\sigma_y$ are defined as follows.

\begin{equation}\label{eq:F-sigma-x}
 \begin{aligned}
 &\sigma_x(\emptyword)=\emptyword,\qquad \sigma_x(1)=0,\\
 &\sigma_x(0u)=00u,\qquad
   \sigma_x(10u)=01u,\qquad \sigma_x(11u)=1u
   \quad(u\in W),
 \end{aligned}
\end{equation}
and
\begin{equation}\label{eq:F-sigma-y}
 \sigma_y(\emptyword)=\emptyword,\qquad
 \sigma_y(0u)=0u,\qquad
 \sigma_y(1u)=1\sigma_x(u)\quad(u\in W).
\end{equation}
Indeed, by Proposition~\ref{prop:F-word-action}, for $t \in \{x,y\}$ and any word $w \in W$ where $P_t(w)$ is well-defined, we have $\sigma_t(w)=P_t(w)$, so $\sigma_t$ simply replaces the prefix of $w$ in the same way as the
map $P_t$. For the two exceptional
words for $x$ in \eqref{eq:F-sigma-x}, $\emptyword$ and $1$, we have
\[x(\kappa(\emptyword))=x(1)=1=\kappa(\emptyword)\] and \[x(\kappa(1))=x(1/2)=1/4=\kappa(0).\]
The equalities \eqref{eq:F-sigma-y} are verified similarly. The words $\emptyword$ and $1$ are fixed because $y$ fixes
$1$ and $1/2$, and for $u\ne\emptyword$, one can use the equality
$\kappa(1u)=(1+\kappa(u))/2$in
\eqref{eq:F-PL-generators}.

Note that the word $\omega \in \FF$  evaluates to the identity element in the representation \eqref{eq:presentation} of $F$ if and only if $\omega(\sigma_x, \sigma_y) = id_W \in \Sym(W)$, and in particular, 
\begin{equation}\label{eq:F-word-relators}
 \rho_1(\sigma_x,\sigma_y)=\id_W \in \Sym(W),
 \qquad \rho_2(\sigma_x,\sigma_y)=\id_W.
\end{equation}

\section{Explicit construction of sofic approximations for $F$.}\label{sec:maps_and_approximations}

In this section we provide an explicit construction of a sequence of sofic approximations for $F$.

\subsection{Finite sets that support sofic approximations}
For $n \geq 1$, we consider the set
\[
 \Omega_n=\left\{(L,R)\mid L:[n] \rightarrow [n] \, \text{and}\, R:[n] \rightarrow [n]\right \}=[n]^{[n]}\times[n]^{[n]}.
\]
We will often call elements of $\Omega_n$ \textit{configurations}. In our subsequent arguments, we will use two convenient additional representations of elements of $\Omega_n$.

First, every configuration $s=(L,R)$ of $\Omega_n$ is associated with a unique rooted directed labeled graph $\Gamma(s)$ with edges labeled by the letters $L$ and $R$ that has the following properties. The vertex set of $\Gamma(s)$ is $[n]$, the root is at $0$, and the directed labeled edges have the form 
$ v\xrightarrow{L}L(v)$ and $ v\xrightarrow{R}R(v)$ for every $v \in [n]$.
Note that the graph $\Gamma(s)$ is allowed to have loops and parallel edges with different labels.

Moreover, every configuration $s=(L,R)$ of $\Omega_n$ can also be viewed as an $n \times 3$ table $T(L,R)$  with rows indexed by the numbers $0 \leq  i \leq n-1$. In $T(L,R)$, the row with index $i$ has entries $[ i, \,L(i), \, R(i)]$.
We simplify this notation to \[\row{i}{L(i),R(i)},\] so the row 
$\row{v}{a,b} $
in $T(L,R)$ means that $L(v)=a$ and $R(v)=b$. In the associated graph $\Gamma(s)$, this row defines
the directed edges $v\xrightarrow{L}a$ and
$v\xrightarrow{R}b$. The second interpretation will be useful for recording certain modifications of elements of $\Omega_n$.

We will denote by $\Prob_n$ the uniform probability measure on $\Omega_n$.

\subsection{Addresses of the numbers and their length.}\label{sec:addresses_and_notation}

Consider a configuration $s=(L,R) \in \Omega_n$, and let us denote by $\langle L,R\rangle$ the submonoid generated by $L$ and $R$ in the monoid of all functions from $[n]$ to itself. Let $\theta_s:\{0,1\}^* \rightarrow \langle L,R \rangle$ be the unique monoid homomorphism from $W =\{0,1\}^*$ to $\langle L,R \rangle $ that sends $0$ to $L$, $1$ to $R$.
\begin{definition}\label{def:address}
  For $a \in [n]$ we will call the \textit{address} of $a$ with respect to $s$ the smallest in the length-lexicographic order\footnote{This is a modification of the lexicographic order where the shorter word is always smaller than a longer one, and words of equal length are compared lexicographically.} word $w \in W$ such that $\theta_s(w)(0)=a$, if such a word exists. If no such word exists, we will say that the address of $a$ is undefined. If the address of $a \in [n]$ is defined, we denote it by $\mu(s,a)$.
\end{definition}
 Moreover, we define the length function $l_s$ with respect to $s$ as follows.

 \begin{definition}\label{def:length}
 For $a \in [n]$ with a well-defined address $w \in W$, we set $l_s(a)=|w|$, where $|w|$ denotes the length of the word $w$. If the address of $a$ is not defined, we set $l_s(a)=\infty$.
 \end{definition}
 
 Note that the address of $0$ is always $\emptyword$ and $l_s(0)=0$ for every $s \in \Omega_n$. Finally, for $ r\geq 0$ we define the ball of radius $r$ centered at $0$ in $s$ as the following subset of $[n]$:
\[B(s,r)=\left\{ a \in [n] \mid l_s(a) \leq r \right\}.\] 

\subsection{Regular subsets of $\Omega_n$} We will use the following special subsets of $\Omega_n$ in our subsequent arguments. 

\begin{definition}
 For $r\geq 0$,  for each $n \in \mathbb{N}$ define
\begin{equation}\label{eq:F-good}
 G_n(r)=\{s\in\Omega_n: \theta_s(w)(0),\ |w|\leq r ,
                              \text{ are pairwise distinct}\}.
\end{equation}   
\end{definition}
A configuration $s=(L,R) \in \Omega_n$ belongs to $G_n(r)$ if and only if the truncated\footnote{By truncated restriction we mean the induced subgraph on $B(s,r)$ with the outgoing edges of vertices at distance $r$ from $0$ removed} restriction of $\Gamma(s)$ to $B(s,r)$ is isomorphic to the complete directed binary rooted tree of depth $r$.

\begin{definition}
 For $s \in \Omega_n$ we will call the depth of $s$, denoted by $d(s)$, the largest integer $r\geq 0$ such that $s \in G_n(r)$.   
\end{definition}

\subsection{The maps \texorpdfstring{$X_n$}{Xn} and \texorpdfstring{$Y_n$}{Yn}}
Now we are ready to define the maps $X_n: \Omega_n \to \Omega_n$ and $Y_n: \Omega_n \to \Omega_n$ that will approximate the generators $x$ and $y$ of $F$ respectively.

\begin{definition}\label{def:X_n}
Define $X_n: \Omega_n \to \Omega_n$ as follows. Consider an arbitrary $z=(L,R)\in\Omega_n$, and set $p=R(0)$.   
\begin{enumerate}
\item If $p\ne0$, $X_n$ will act on $(L,R)$ by rewriting the rows of the table $T(L,R)$ indexed by $0$ and $p$ as follows:
\begin{equation}\label{eq:X_n}
 \bigl(\row{0}{a,p},\ \row{p}{b,c}\bigr)
 \ \xmapsto{\qquad}\
 \bigl(\row{0}{p,c},\ \row{p}{a,b}\bigr),
\end{equation}
while the rest of the rows of $T(L,R)$ are unchanged. It is easy to see that we still get an admissible table that corresponds to some $(L',R') \in \Omega_n$.
The corresponding operation with $\Gamma(z)$ is illustrated by Figure~\ref{fig:x}. We will refer to this case as the \textit{regular case}.

\item If $p=0$, $X_n$ will simply interchange the two entries in the row indexed by $0$:
\begin{equation}\label{eq:X-exception}
 \row{0}{a,0}\longmapsto\row{0}{0,a},
\end{equation}
and all other rows remain unchanged. The corresponding transformation of $\Gamma(z)$ is illustrated by Item (a) in Figure~\ref{fig:exceptional}. We will refer to this case as the \textit{exceptional case}.
Note that if $L(0)=R(0)=0$, the configuration $z=(L,R)$ is unchanged.
\end{enumerate}
\end{definition}

\begin{remark}\label{rem:Xn_and_sigma_x}
  Note that for any $z \in \Omega_n$ and any $t \in [n]$, an application of $X_n$ to $z$ changes $L(t)$ or $R(t)$ only at the points $t$ with $l_z(t) \leq 2$. Moreover, if $d(z) \geq 2$ and $z'=X_n(z)$, we are in the regular case of the definition of $X_n$, and for $t \in \{0, p, a,b,c\}$ its address changes by an application of $\sigma_x$ defined in Section~\ref{sec:fw_action}, or more explicitly:
\[\mu(z',t)=\sigma_x(\mu(z,t)).\]
\end{remark}

\begin{definition}\label{def:Y_n}
To define the map $Y_n$, for any $z=(L,R) \in \Omega_n$, we examine the values of $R(0)$ and $R^2(0)$ and consider the following cases.
\begin{enumerate}
    \item If $R(0)=0$, set $Y_n(z)=z$.
    \item If $R(0)\ne0$, set $r=R(0)$, and consider $p=R(r)=R^2(0)$. If we have $p\notin\{0,r\}$, then $Y_n$ changes the rows of $T(L,R)$ indexed by $0$, $r$, and $p$ as follows
    \begin{equation}\label{eq:Y_n}
 \bigl(\row{0}{s,r},\ \row{r}{a,p}, \ p:(b,c)\bigr)
 \ \xmapsto{\qquad}\
 \bigl(\row{0}{s,r},\ \row{r}{p,c},\ \row{p}{a,b}\bigr).
 \end{equation} 
  The corresponding modification of $\Gamma(z)$ is described in Figure~\ref{fig:y}.
  If $p \in \{0,r\} $, we proceed to the next case.
  \item If $r=R(0)\ne0$, but $p=R^2(0) \in\{0,r\}$, then $Y_n$ will interchange the two entries in the row labeled by $r$:
\begin{equation}\label{eq:Y-exception}
 \row{r}{a,p}\longmapsto\row{r}{p,a}.
 \end{equation}
 The corresponding transformation of $\Gamma(z)$ is illustrated by Items $(b)$ and $(c)$ in Figure~\ref{fig:exceptional}. 
\end{enumerate}
Note that the rows that are not mentioned always stay unchanged. We will also refer to cases $(1)$ and $(3)$ as the \textit{exceptional cases} for $Y_n$.
\end{definition}

\begin{remark}\label{rem:Yn_and_sigma_y}
  Similarly to $X_n$, for any $z \in \Omega_n$ and any $t \in [n]$, an application of $Y_n$ to $z$ changes $L(t)$ or $R(t)$ only at the points $t$ with $l_z(t) \leq 2$. Moreover, if $d(z) \geq 4$ and $z'=Y_n(z)$, we are in the regular case of the definition of $Y_n$, and for $t \in \{0,r, p, s,a,b,c\}$ its address changes by an application of $\sigma_y$ defined in Section~\ref{sec:fw_action}, or more explicitly:
\[\mu(z',t)=\sigma_y(\mu(z,t)).\]
\end{remark}

Figures~\ref{fig:x}--\ref{fig:exceptional} illustrate the definitions of $X_n$ and $Y_n$. In each figure the relevant part of $\Gamma(z)$ is on the left, and its modified version is on the right.

\tikzset{
  ptr vertex/.style={circle,draw,minimum size=5.5mm,inner sep=1pt,
    font=\small},
  ptr changed/.style={ptr vertex,fill=black!9},
  ptr root/.style={ptr vertex,double,double distance=0.65pt},
  ptr root changed/.style={ptr root,fill=black!9},
  ptr L/.style={-{Stealth[length=2mm]},draw=blue!65!black,
    line width=0.65pt},
  ptr R/.style={-{Stealth[length=2mm]},draw=red!65!black,
    line width=0.65pt,densely dashed},
  ptr label/.style={font=\footnotesize,fill=white,inner sep=1.1pt},
  ptr operation/.style={-{Stealth[length=2.5mm]},line width=0.8pt}
}

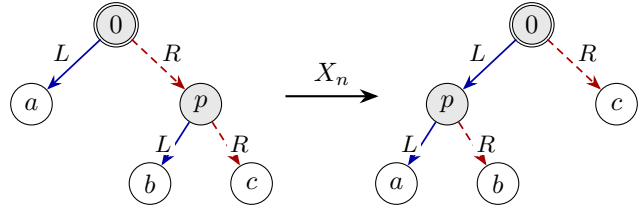
\begin{figure}[htbp]
\centering
\begin{tikzpicture}[x=1.12cm,y=1.05cm]
  \begin{scope}
    \node[ptr root changed] (x0) at (0,2) {$0$};
    \node[ptr vertex] (xa) at (-1,1) {$a$};
    \node[ptr changed] (xp) at (1,1) {$p$};
    \node[ptr vertex] (xb) at (0.4,0) {$b$};
    \node[ptr vertex] (xc) at (1.6,0) {$c$};
    \draw[ptr L] (x0) -- node[ptr label,above left] {$L$} (xa);
    \draw[ptr R] (x0) -- node[ptr label,above right] {$R$} (xp);
    \draw[ptr L] (xp) -- node[ptr label,left] {$L$} (xb);
    \draw[ptr R] (xp) -- node[ptr label,right] {$R$} (xc);
  \end{scope}
  \draw[ptr operation] (2.0,1.1) -- node[above,font=\small] {$X_n$} (3.1,1.1);
  \begin{scope}[xshift=5.5cm]
    \node[ptr root changed] (z0) at (0,2) {$0$};
    \node[ptr changed] (zp) at (-1,1) {$p$};
    \node[ptr vertex] (zc) at (1,1) {$c$};
    \node[ptr vertex] (za) at (-1.6,0) {$a$};
    \node[ptr vertex] (zb) at (-0.4,0) {$b$};
    \draw[ptr L] (z0) -- node[ptr label,above left] {$L$} (zp);
    \draw[ptr R] (z0) -- node[ptr label,above right] {$R$} (zc);
    \draw[ptr L] (zp) -- node[ptr label,left] {$L$} (za);
    \draw[ptr R] (zp) -- node[ptr label,right] {$R$} (zb);
  \end{scope}
\end{tikzpicture}
\caption{Regular case of the definition of $X_n$, with $p=R(0)\ne0$. The graph $\Gamma(z)$ is on the left and the graph $\Gamma(X_n(z))$ is on the right.
The vertices whose images under $L$ or $R$ might change are shaded, and the double circle marks the distinguished root. Outgoing edges that are not changing are omitted.
Our only restriction on the repetitions in the list $a,b,c,0,p$ is $p\ne0$, so for example, it is possible to have $a=c$. We also have no restrictions on other outgoing edges.}
\label{fig:x}
\end{figure}

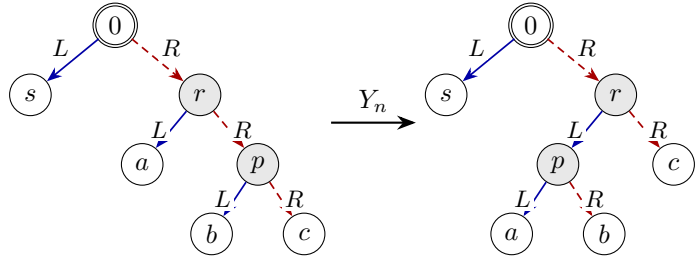
\begin{figure}[htbp]
\centering
\begin{tikzpicture}[x=1.02cm,y=0.92cm]
  \begin{scope}
    \node[ptr root] (y0) at (0,3) {$0$};
    \node[ptr vertex] (ys) at (-1.1,2) {$s$};
    \node[ptr changed] (yr) at (1.1,2) {$r$};
    \node[ptr vertex] (ya) at (0.35,1) {$a$};
    \node[ptr changed] (yp) at (1.85,1) {$p$};
    \node[ptr vertex] (yb) at (1.25,0) {$b$};
    \node[ptr vertex] (yc) at (2.45,0) {$c$};
    \draw[ptr L] (y0) -- node[ptr label,above left] {$L$} (ys);
    \draw[ptr R] (y0) -- node[ptr label,above right] {$R$} (yr);
    \draw[ptr L] (yr) -- node[ptr label,left] {$L$} (ya);
    \draw[ptr R] (yr) -- node[ptr label,right] {$R$} (yp);
    \draw[ptr L] (yp) -- node[ptr label,left] {$L$} (yb);
    \draw[ptr R] (yp) -- node[ptr label,right] {$R$} (yc);
  \end{scope}
  \draw[ptr operation] (2.8,1.6) -- node[above,font=\small] {$Y_n$} (3.9,1.6);
  \begin{scope}[xshift=5.5cm]
    \node[ptr root] (w0) at (0,3) {$0$};
    \node[ptr vertex] (ws) at (-1.1,2) {$s$};
    \node[ptr changed] (wr) at (1.1,2) {$r$};
    \node[ptr changed] (wp) at (0.35,1) {$p$};
    \node[ptr vertex] (wc) at (1.85,1) {$c$};
    \node[ptr vertex] (wa) at (-0.25,0) {$a$};
    \node[ptr vertex] (wb) at (0.95,0) {$b$};
    \draw[ptr L] (w0) -- node[ptr label,above left] {$L$} (ws);
    \draw[ptr R] (w0) -- node[ptr label,above right] {$R$} (wr);
    \draw[ptr L] (wr) -- node[ptr label,left] {$L$} (wp);
    \draw[ptr R] (wr) -- node[ptr label,right] {$R$} (wc);
    \draw[ptr L] (wp) -- node[ptr label,left] {$L$} (wa);
    \draw[ptr R] (wp) -- node[ptr label,right] {$R$} (wb);
  \end{scope}
\end{tikzpicture}
\caption{Regular case of the definition of $Y_n$, with
$r=R(0)\ne0$ and $p=R(r)\notin\{0,r\}$.
The graph $\Gamma(z)$ is on the left and the graph $\Gamma(Y_n(z))$ is on the right.
The vertices whose images under $L$ or $R$ might change are shaded, and the double circle marks the distinguished root. Outgoing edges that are not changing are omitted.
We also do not impose any additional restrictions on the values of $s,a,b$, and $c$, and have no restrictions on other outgoing edges.}
\label{fig:y}
\end{figure}

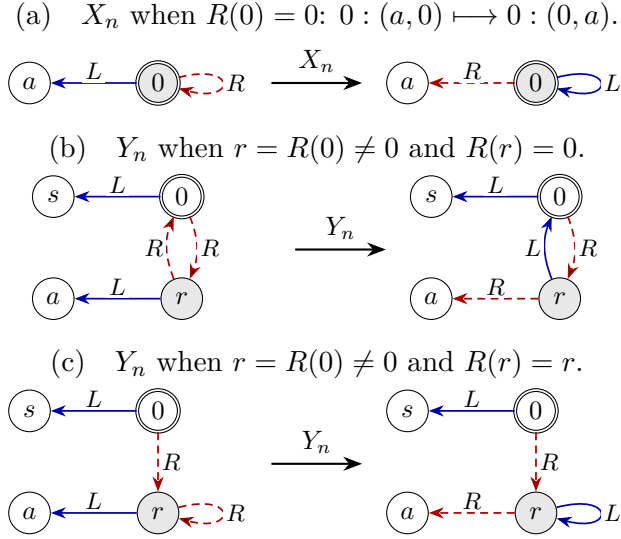
\begin{figure}[htbp]
\centering
\textup{(a)}\quad $X_n$ when $R(0)=0$: $0:(a,0)\longmapsto0:(0,a)$.
\par\smallskip
\begin{tikzpicture}[x=1cm,y=1cm]
  \node[ptr root changed] (e0) at (0,0) {$0$};
  \node[ptr vertex] (ea) at (-1.7,0) {$a$};
  \draw[ptr L] (e0) -- node[ptr label,above] {$L$} (ea);
  \path[ptr R] (e0) edge[loop right,min distance=8mm]
    node[ptr label,right] {$R$} (e0);
  \draw[ptr operation] (1.5,0) -- node[above,font=\small] {$X_n$} (2.7,0);
  \node[ptr root changed] (f0) at (5.0,0) {$0$};
  \node[ptr vertex] (fa) at (3.3,0) {$a$};
  \draw[ptr R] (f0) -- node[ptr label,above] {$R$} (fa);
  \path[ptr L] (f0) edge[loop right,min distance=8mm]
    node[ptr label,right] {$L$} (f0);
\end{tikzpicture}
\par\medskip
\textup{(b)}\quad $Y_n$ when $r=R(0)\ne0$ and $R(r)=0$.
\par\smallskip
\begin{tikzpicture}[x=1cm,y=1cm]
  \node[ptr root] (g0) at (0,1.35) {$0$};
  \node[ptr vertex] (gs) at (-1.7,1.35) {$s$};
  \node[ptr changed] (gr) at (0,0) {$r$};
  \node[ptr vertex] (ga) at (-1.7,0) {$a$};
  \draw[ptr L] (g0) -- node[ptr label,above] {$L$} (gs);
  \path[ptr R] (g0) edge[bend left=22]
    node[ptr label,right] {$R$} (gr);
  \path[ptr R] (gr) edge[bend left=22]
    node[ptr label,left] {$R$} (g0);
  \draw[ptr L] (gr) -- node[ptr label,above] {$L$} (ga);
  \draw[ptr operation] (1.5,0.65) -- node[above,font=\small] {$Y_n$} (2.7,0.65);
  \node[ptr root] (h0) at (5.0,1.35) {$0$};
  \node[ptr vertex] (hs) at (3.3,1.35) {$s$};
  \node[ptr changed] (hr) at (5.0,0) {$r$};
  \node[ptr vertex] (ha) at (3.3,0) {$a$};
  \draw[ptr L] (h0) -- node[ptr label,above] {$L$} (hs);
  \path[ptr R] (h0) edge[bend left=22]
    node[ptr label,right] {$R$} (hr);
  \path[ptr L] (hr) edge[bend left=22]
    node[ptr label,left] {$L$} (h0);
  \draw[ptr R] (hr) -- node[ptr label,above] {$R$} (ha);
\end{tikzpicture}
\par\medskip
\textup{(c)}\quad $Y_n$ when $r=R(0)\ne0$ and $R(r)=r$.
\par\smallskip
\begin{tikzpicture}[x=1cm,y=1cm]
  \node[ptr root] (i0) at (0,1.35) {$0$};
  \node[ptr vertex] (is) at (-1.7,1.35) {$s$};
  \node[ptr changed] (ir) at (0,0) {$r$};
  \node[ptr vertex] (ia) at (-1.7,0) {$a$};
  \draw[ptr L] (i0) -- node[ptr label,above] {$L$} (is);
  \draw[ptr R] (i0) -- node[ptr label,right] {$R$} (ir);
  \draw[ptr L] (ir) -- node[ptr label,above] {$L$} (ia);
  \path[ptr R] (ir) edge[loop right,min distance=8mm]
    node[ptr label,right] {$R$} (ir);
  \draw[ptr operation] (1.5,0.65) -- node[above,font=\small] {$Y_n$} (2.7,0.65);
  \node[ptr root] (j0) at (5.0,1.35) {$0$};
  \node[ptr vertex] (js) at (3.3,1.35) {$s$};
  \node[ptr changed] (jr) at (5.0,0) {$r$};
  \node[ptr vertex] (ja) at (3.3,0) {$a$};
  \draw[ptr L] (j0) -- node[ptr label,above] {$L$} (js);
  \draw[ptr R] (j0) -- node[ptr label,right] {$R$} (jr);
  \draw[ptr R] (jr) -- node[ptr label,above] {$R$} (ja);
  \path[ptr L] (jr) edge[loop right,min distance=8mm]
    node[ptr label,right] {$L$} (jr);
\end{tikzpicture}
\caption{The exceptional case for $X_n$ is covered by Item $(a)$, and the non-trivial exceptional cases for $Y_n$ are described by Items $(b)$ and $(c)$.}
\label{fig:exceptional}
\end{figure}

\FloatBarrier

\begin{lemma}\label{lem:bijections}
The maps \(X_n\) and \(Y_n\) are permutations of \(\Omega_n\).
\end{lemma}
\begin{proof}
Recall from Definition~\ref{def:X_n} that, for $u=(L,R)\in\Omega_n$
and $p=R(0)$, the map $X_n:\Omega_n\to\Omega_n$ acts as follows.
If $p\ne0$, it rewrites the rows indexed by $0$ and $p$ by
\begin{equation*}
 \bigl(\row{0}{a,p},\ \row{p}{b,c}\bigr)
 \longmapsto
 \bigl(\row{0}{p,c},\ \row{p}{a,b}\bigr),
 \tag{4.2}
\end{equation*}
where $a=L(0)$, $b=L(p)$, and $c=R(p)$.
If $p=0$, it interchanges the entries of row $0$:
\begin{equation*}
 \row{0}{a,0}\longmapsto\row{0}{0,a}.
 \tag{4.3}
\end{equation*}
All other rows remain unchanged.

Similarly, by Definition~\ref{def:Y_n}, the map
$Y_n:\Omega_n\to\Omega_n$ fixes $u$ if $R(0)=0$.
Otherwise, put $r=R(0)$ and $p=R(r)$.
If $p\notin\{0,r\}$, the rule is
\begin{equation*}
 \bigl(\row{0}{s,r},\ \row{r}{a,p},\ \row{p}{b,c}\bigr)
 \longmapsto
 \bigl(\row{0}{s,r},\ \row{r}{p,c},\ \row{p}{a,b}\bigr),
 \tag{4.4}
\end{equation*}
where $s=L(0)$, $a=L(r)$, $b=L(p)$, and $c=R(p)$.
If $p\in\{0,r\}$, it interchanges the entries of row $r$:
\begin{equation*}
 \row{r}{a,p}\longmapsto\row{r}{p,a}.
 \tag{4.5}
\end{equation*}
Rows not mentioned remain unchanged.

Since $\Omega_n$ is finite, it suffices to prove that $X_n$ and
$Y_n$ are injective. Let $z_1=(L_1,R_1)$ and $z_2=(L_2,R_2)$
be two distinct elements of $\Omega_n$.

\medskip
\noindent\textbf{The $X_n$ case.}
Put $p_i=R_i(0)$ and write $X_n(z_i)=(L_i',R_i')$ for $i=1,2$.
In both branches, $L_i'(0)=p_i$. We consider three cases.
\begin{enumerate}
\item $p_1\ne p_2=0$, after interchanging the inputs if necessary.
Then $L_1'(0)=p_1\ne0=L_2'(0)$, so $X_n(z_1)\ne X_n(z_2)$.

\item $p_1=p_2=0$.
For both inputs, $X_n$ interchanges the entries of row $0$ and leaves all other
rows unchanged. A difference in $L_1(0),L_2(0)$ becomes a difference
in $R_1'(0),R_2'(0)$, and a difference in any other row persists.
Thus $X_n(z_1)\ne X_n(z_2)$.

\item $p_1\ne0$ and $p_2\ne0$.
If $p_1\ne p_2$, then $L_1'(0)\ne L_2'(0)$, so the outputs are
distinct. If $p_1=p_2=p$, the regular rule permutes the four entries
in rows $0$ and $p$ in the same way for both configurations, leaving
every other entry unchanged. Thus, as in case~(2), it preserves any
difference between the inputs, and $X_n(z_1)\ne X_n(z_2)$.
\end{enumerate}

\medskip
\noindent\textbf{The $Y_n$ case.}
Put $r_i=R_i(0)$ and write $Y_n(z_i)=(L_i',R_i')$ for $i=1,2$.
The root row is unchanged, so $R_i'(0)=r_i$.
\begin{enumerate}
\item $r_1\ne r_2=0$, after interchanging the inputs if necessary.
Then $R_1'(0)=r_1\ne0=R_2'(0)$, so $Y_n(z_1)\ne Y_n(z_2)$.

\item $r_1=r_2=0$.
Both inputs are fixed, so $Y_n(z_1)=z_1\ne z_2=Y_n(z_2)$.

\item $r_1\ne0$ and $r_2\ne0$.
If $r_1\ne r_2$, the outputs differ in their right root pointers.
Otherwise, write $r_1=r_2=r$ and put $p_i=R_i(r)$.
Since $L_i'(r)=p_i$, distinct $p_1,p_2$ give distinct outputs.
If $p_1=p_2=p\in\{0,r\}$, both inputs undergo the same swap in
row $r$, with all other rows unchanged, so any difference persists.
Finally, if $p_1=p_2=p\notin\{0,r\}$, the argument in case~(3)
for $X_n$ applies with $r$ in place of $0$.
Thus $Y_n(z_1)\ne Y_n(z_2)$ in every subcase.
\end{enumerate}
Both maps are therefore injective and hence are permutations of $\Omega_n$.
\end{proof}

It is easy to see that if $z \in \Omega_n$ belongs to the regular case for some $t \in \{X_n ,Y_n\}$ and $z'=t(z)$, then the modification of $\Gamma(z')$ under the action of $t^{-1}$ is described by switching the order of trees in Figure~\ref{fig:x} or Figure~\ref{fig:y} depending on the value of $t$. We will also say that $z'$  belongs to the regular case for $t^{-1}$ if $z$ belongs to the regular case of $t$.

\subsection{Partial action of $\langle X_n, Y_n \rangle$ on $G_n(r)$}

We will use the following conventions throughout this section. For $t$ in $\{x^{\pm1},y^{\pm1}\}$ we denote by $t_n$ the corresponding permutation in $\{X_n, Y_n, X_n^{-1}, Y_n^{-1}\}$. Note that for any $z \in G_n(3)$ and any permutation
$t_n \in \{X_n, Y_n, X_n^{-1}, Y_n^{-1}\}$, the configuration $z$ belongs to the regular case of the definition of $t_n$, and the configuration $t_n(z)$ belongs to the regular case for $t_n^{-1}$. If $t_n =X_n^{\pm 1}$, we will retain the notation from Figure \ref{fig:x}, and we will call $a,b$, and $c$ \textbf{the leaves of} $t_n$ in $z$, and the numbers $0$ and $p$ will be called \textbf{the internal vertices} of $t_n$ in $z$. Similarly, if $t_n=Y_n^{\pm 1}$ we will adhere to the notation of Figure \ref{fig:y} and the numbers $s,a,b$, and $c$ will be called \textbf{the leaves of} $t_n$, while the numbers $0$, $r$, and $p$ will be called \textbf{the internal vertices of} $t_n$ in $z$. 

Finally, for any $z \in \Omega_n$, by a directed path $\gamma$  of length $k$ in $\Gamma(z)$ we mean a sequence of directed edges $\gamma=(e_1, \ldots, e_k)$  such that for $i=1, \ldots, k-1$ the end vertex of $e_i$ is equal to the initial vertex of $e_{i+1}$.  A directed path $\gamma$ of length $k$ in $\Gamma(z)$ is uniquely determined by its starting vertex and the sequence of length $k$ of symbols from $\{L,R\}$ that label its edges. Thus, for any $z \in \Omega_n$, the directed paths in $\Gamma(z)$ that start at $0$ are in one-to-one correspondence with the binary words over the alphabet $\{L,R\}$. Moreover, we can naturally identify them with the words in $W=\{0,1\}^*$ by coding $L$ by $0$ and $R$ by $1$. For a path $\gamma$ starting at $0$, we will denote the corresponding word in $W=\{0,1\}^*$ by $\mu(\gamma)$.

\begin{lemma}\label{lem:xy_on_good_paths}
Consider $z \in G_n(4)$ and let $t_n$ be any permutation in $\{X_n, Y_n, X_n^{-1}, Y_n^{-1}\}$. Let us also denote $z'=t_n(z)$. Consider a path $\gamma$ in $\Gamma(z)$ that starts at $0$, visits the set of leaves of $t_n$ in $z$, and never returns to the set of internal vertices of $t_n$ in $z$ afterwards. Suppose that $h$ is the first leaf of $t_n$ visited by $\gamma$. Then the following statements hold.
\begin{enumerate}
    \item The coding $\mu(\gamma)$ has the form \[\mu(\gamma)=\mu(z, h)\beta,\] for some $\beta \in W$.
    \item Consider the path in $\Gamma(z')$ starting at $0$ and coded by the word \[\sigma_t(\mu(z,h))\beta=\mu(z',h)\beta.\] Then, after reaching $h$, this path visits exactly the same numbers from $[n]$ as $\gamma$ in the same order as $\gamma$.
\end{enumerate}
 
\end{lemma}

\begin{proof}
Note that since $z \in G_n(4)$, it is covered by the regular case of the definition for every $t_n \in \{X_n,Y_n,X_n^{-1},Y_n^{-1}\}$.

Item (1) is obvious - a directed path $\gamma$ starting at $0$ leaves the internal vertices of $t_n$ in $z$ for the first time only by moving along a directed path of minimal length that connects $0$ with one of the leaves of $t_n$ in $z$. This leaf is $h$, and the corresponding minimal path is the address of $h$, because $d(z) \geq 3$.

To prove Item $(2)$, first note that for $t_n=X_n^{\pm1}$, we have $\mu(z',h)=\sigma_t(\mu(z,h))$ by  Remark~\ref{rem:Xn_and_sigma_x}, and for $t_n=Y_n^{\pm1}$, we get $\mu(z',h)=\sigma_t(\mu(z,h))$ by Remark~\ref{rem:Yn_and_sigma_y} and its adaptation\footnote{A careful inspection shows that we need the assumption $z \in G_n(4)$ in this specific case to prevent an issue with addresses when there is an edge from $a$ to $b$.} for $Y_n^{-1}$. If $z=(L,R)$ and $z'=(L',R')$, then for any $m  \in [n]$ we can have $L(m)\ne L'(m)$ or $R(m) \ne R'(m)$ only if $m$ is one of the internal vertices of $t_n$ in $z$. In particular, $L(h)=L'(h)$ and $R(h)=R'(h)$. Since by our assumption $\gamma$ never visits an internal vertex after reaching $h$ for the first time, we have $L(m)=L'(m)$ and $R(m)=R'(m)$ for every number $m$ visited by $\gamma$ after $h$. Thus, the path coded by $\mu(z',h)\beta$ in $\Gamma(z')$, after reaching $h$ for the first time, visits exactly the same numbers from $[n]$ as $\gamma$ in the same order as $\gamma$.
\end{proof}

\begin{lemma}\label{lem:regularity persists}
Let $t_n$ be any permutation in $\{X_n, Y_n, X_n^{-1}, Y_n^{-1}\}$. Suppose that $z \in \Omega_n$ satisfies the inequality $d(z) \geq 5$ and set $z'=t_n(z)$. Then we have \[d(z') \geq d(z)-1.\]
\end{lemma}

\begin{proof}
Since $d(z) \geq 5$, we know that $z$ belongs to the regular case in the definition of $t_n \in \{X_n, Y_n, X_n^{-1}, Y_n^{-1}\}$. To simplify the notation, let us denote $d(z)=d \geq 5$. We are going to show that $d(z')\geq d-1$.

For a non-negative $k \in \mathbb{Z}$, let us denote by $W_k$ the set of binary words of length at most $k$ in $W$. It suffices to check that all directed paths started at $0$ and coded by the words from $W_{d-1}$ in $\Gamma(z')$ have distinct endpoints in $\Gamma(z')$. 
Since $d=d(z)$, we already know that all directed paths started at $0$ and coded by the words from $W_d$ in $\Gamma(z)$ have distinct endpoints in $\Gamma(z)$.

Moreover, by the definition of $\sigma_t$, for any word $\mu \in W$ we have $||\sigma_t(\mu)| -|\mu|| \leq 1$. Thus, \[\sigma_t^{-1}(W_{d-1}) \subseteq W_d.\]

Let $L\subseteq W_d$ be the list of addresses of the leaves and internal vertices of $t_n$ in $z$. Since we are in the regular case of the definition of $t_n$, for any $\mu \in L$, the path in $\Gamma(z')$ started at $0$ and coded by the word $\sigma_t(\mu)$ ends at one of the leaves or internal vertices of $t_n^{-1}$ in $z'$, and its endpoint is the same number as the endpoint of $\mu$.

Observe that for any $\mu \in W_d \setminus L$, the path coded by $\mu$ in $\Gamma(z)$ visits some leaf $h$ of $t_n$ in $z$, and then never returns to the internal vertices of $t_n$ in $z$. Thus, since $d \geq 5$, we may apply Lemma~\ref{lem:xy_on_good_paths} to any $\mu \in W_d \setminus L$.  This leads to the decomposition \[\mu =\mu(z,h)\beta\] for the first visited leaf $h$, such that the path in $\Gamma(z')$ coded by $\sigma_t(\mu(z,h))\beta$ ends at the same number as $\mu$. But since $P_t$ is defined on $\mu(z,h)$ and $\sigma_t$ extends $P_t$, we have \[\sigma_t(\mu(z,h))\beta=\sigma_t(\mu(z,h)\beta)=\sigma_t(\mu).\]

Thus, for $\mu \in W_d \setminus L$, the path with the code $\mu$ in $\Gamma(z)$ and the path with the code $\sigma_t(\mu)$ in $\Gamma(z')$ both end at the same number from $[n]$. But this implies that all paths coded by the words from $\sigma_t(W_d)$ end at different points in $\Gamma(z')$. Since $W_{d-1} \subseteq \sigma_t(W_d)$, we conclude that $d(z') \geq d-1$.
\end{proof}

\begin{lemma}\label{lem:length_control}

Consider an arbitrary $z$ with $d(z) \geq 5$ and any number $m \in [n]$. Let $t_n$ be any permutation in $\{X_n, Y_n, X_n^{-1}, Y_n^{-1}\}$ and set $z'=t_n(z)$. Then the lengths $l_z(m)$ and $l_{z'}(m)$ differ by at most $1$:
\[|l_z(m)-l_{z'}(m)| \leq 1.\]
\end{lemma}

\begin{proof}
Note that by assumption, $z$ belongs to the regular case of the definition of $t_n$. It is easy to see that $l_z(m)$ and $l_{z'}(m)$ are equal to the length of the shortest directed path from $0$ to $m$ in $\Gamma(z)$ and in $\Gamma(z')$ respectively. Moreover, the graph $\Gamma(z)$ has a directed path from $0$ to $m$ if and only if the graph $\Gamma(z')$ has a directed path from $0$ to $m$. Thus, $l_z(m)$ and $l_{z'}(m)$ are both finite or both infinite simultaneously.

The inequality obviously holds if both $l_z(m)$ and $l_{z'}(m)$ are infinite. Moreover, the definition of $X_n$ and $Y_n$ implies that we have $|l_z(m)-l_{z'}(m)| \leq 1$ when $m$ is an internal vertex or a leaf of $t_n$ in $z$.
If both $l_z(m)$ and $l_{z'}(m)$ are finite and $m$ is not an internal vertex or a leaf of $t_n$ in $z$, the conclusion will follow from two inequalities  \begin{equation}\label{eq:len_incr_1}
    l_{z'}(m)+1 \geq l_z(m),
\end{equation}
and 
\begin{equation}\label{eq:len_incr_2}
    l_{z}(m)+1 \geq l_{z'}(m).
\end{equation}
We first prove the inequality~\eqref{eq:len_incr_1}. Suppose that $\gamma$ is the shortest directed path from $0$ to $m$ in $\Gamma(z')$. Then $\gamma$ visits a leaf of $t_n^{-1}$ in $z'$. Moreover, since $d(z') \geq 4$ and $\gamma$ has minimal length, it does not visit any internal vertex of $t_n^{-1}$ in $z'$ afterwards.  Let $h$ be the first leaf visited by $\gamma$. Then, the coding $\mu(\gamma)$ has the form \[\mu(\gamma)=\mu(z',h)\beta\] for some $\beta \in W$. By Item $2$ of Lemma~\ref{lem:xy_on_good_paths} applied to $z'$, $t_n^{-1}$ and $\gamma$, the path in $\Gamma(z)$ coded by the word \[\sigma_t^{-1}(\mu(z',h))\beta=\mu(z,h)\beta\] starts at $0$ and ends at $m$. Hence, 
\[|\mu(z,h)\beta|=|\mu(z,h)|+|\beta| \geq l_z(m)\]
But then we have 
\[l_{z'}(m)=|\gamma|=|\mu(z',h)|+|\beta|\geq|\mu(z,h)|-1+|\beta|\geq l_z(m)-1.\]
The inequality \eqref{eq:len_incr_2} is proved in the same way, but with the roles of $z$ and $z'$ switched. Note that the inequality $d(z)\geq 5$ is already assumed in the statement of the lemma.
\end{proof}

\begin{lemma}\label{lem:top_induction_base}
Let $z \in \Omega_n$ be a configuration that satisfies the inequality $d(z) \geq 5$. Consider any $t_n \in \{X_n, Y_n, X_n^{-1}, Y_n^{-1}\}$ and set $z'=t_n(z)$. Then for any $m \in [n]$ with $l_z(m) \leq d(z) - 2$, we have 
\[\mu(z',m)=\sigma_t(\mu(z,m)).\]
\end{lemma}

\begin{proof}
The statement of this lemma is obviously true if $m$ is one of the internal vertices or leaves of $t_n$, so we  may assume that $m$ does not belong to any of these sets. Then, any path\footnote{Since the statement of this lemma implies that $l_z(m)$ is finite, we may assume that such a path exists, and the address of $m$ is well-defined.} from $0$ to $m$ in $\Gamma(z)$ visits the set of the leaves of $t_n$ in $z$.

Since $l_z(m) \leq d(z)$, there is a unique directed path $\gamma$ of length at most $d(z)$ from $0$ to $m$. Then, this path is coded by the address of $m$ in $z$, so we have
\[\mu(\gamma)=\mu(z,m)\]
in $\Gamma(z)$.

Moreover, by Lemma \ref{lem:length_control}, we get 
\[l_{z'}(m)\leq l_z(m)+1 \leq d(z)-1.\]

Notice that Lemma \ref{lem:regularity persists} implies that
\[d(z')\geq d(z) -1, \]
so we also have $l_{z'}(m) \leq d(z')$. But then there exists exactly one path $\gamma'$ of length at most $d(z')$ from $0$ to $m$ in $\Gamma(z')$, and this path is coded by the address of $m$ in $z'$, so 
\[\mu(\gamma')=\mu(z',m).\]
Now we will compare a new path in $\Gamma(z')$ constructed from $\gamma$ with $\gamma'$. Let $h$ be the first leaf of $t_n$ visited by $\gamma$ in $\Gamma(z)$. It is easy to see that we have
\[\mu(\gamma)=\mu(z,h)\beta\] for some $\beta \in W$. Moreover, since $|\gamma| \leq d(z)$, $\gamma$ never visits the internal vertices of $t_n$ after visiting $h$, so by Item (2) of Lemma~\ref{lem:xy_on_good_paths}, the directed path in $\Gamma(z')$ coded by $\sigma_t(\mu(z,h))\beta$ also starts at $0$ and ends at $m$. Furthermore, its length does not exceed $|\gamma|+1$, and since 
\[|\gamma|+1=l_z(m)+1\leq d(z)-1\leq d(z'),\]

this path must coincide with $\gamma'$.
Thus, we have the desired equality
\[\mu(z',m)=\sigma_t(\mu(z,h))\beta=\sigma_t(\mu(z,m)),\]

where the last equality follows from the fact that $\sigma_t$ extends the action of $P_t$ on $W$.

\end{proof}

\begin{lemma}\label{lem:no_outchanges}
Let $\rho =w_k \ldots w_1 \in \FF$ be an arbitrary element of the free group with length $k$. Consider arbitrary $z \in \Omega_n$ with $d(z) \geq k+6$ and define $z_k=(L_k,R_k)=\rho(X_n,Y_n)(z)$. Then for any $m \in [n]$ with $l_z(m) \geq k+3$, we have \[L_k(m)=L(m) \text{ and } R_k(m)=R(m).\]
\end{lemma}

\begin{proof}
The case when $ l_z(m) = \infty$ is obvious. Thus, we may assume that there is a directed path from $0$ to $m$ in $\Gamma(z)$.

For $i=1, \ldots k$, set $\rho_i=w_i \ldots w_1(X_n,Y_n)$ - the word $w_i...w_1$ evaluated over $X_n,Y_n$, and let $z_i=(L_i,R_i)=\rho_i(z)$.

A straightforward induction on $i$ together with Lemma~\ref{lem:regularity persists} shows that for each $i=1, \ldots k$, we have 
\[d(z_i) \geq k+6-i.\]

Then, induction on $i$ together with Lemma~\ref{lem:length_control} shows that for each $i=1, \ldots k$, we have 
\[l_{z_i}(m) \geq k+3-i.\]

Since $l_{z_i}(m) \geq 3$ for every 
$i=1, \ldots k$, the definitions of $X_n$ and $Y_n$ imply that $L_i(m)$ and $R_i(m)$ always stay fixed, so we indeed have $L_k(m)=L(m)$ and $R_k(m)=R(m)$.
\end{proof}

\begin{lemma}\label{lem:addr_at_the_top}
Let $\rho =w_k \ldots w_1 \in \FF$ be any element of the length $k$ in the free group. Consider any $z \in \Omega_n$ with $d(z) \geq 3k+6$ and set $z_k=(L_k,R_k)=\rho(X_n,Y_n)(z)$.
Then for every $m$ with $l_z(m) \leq k+4$, we have \[\mu(z_k,m)=\sigma_{\pi(\rho)}(\mu(z,m)).\]
\end{lemma}

\begin{proof}
For $i=1, \ldots k$, set $\rho_i(X_n,Y_n)=w_i \ldots w_1(X_n,Y_n)$ -- the word $w_i...w_1$ evaluated over $X_n,Y_n$, and let $z_i=(L_i,R_i)=\rho_i(X_n,Y_n)(z)$ with $z_0=z$.
Let us also define $\pi(\rho_i)=w_i \ldots w_1(x,y)$ -- 
the word $w_i...w_1$ evaluated over $x,y$ in the Thompson's group $F$. Finally, to simplify the notation, for $i=1, \ldots, k$, set $\sigma_i=\sigma_{\pi(\rho_i)}$.

A straightforward induction on $i$ together with Lemma~\ref{lem:regularity persists} shows that for each $i=1, \ldots k$, we have 
\[d(z_i) \geq 3k+6-i \geq 2k+6.\]

Moreover, induction on $i$ together with Lemma~\ref{lem:length_control} shows that for each $i=1, \ldots k$, we have 
\[l_{z_i}(m) \leq k+4+i \leq 2k+4 \leq d(z_i)-2.\]

Thus, for each $i=0, \ldots k-1$, the number $m$ and the configuration $z_i$ satisfy the conditions of Lemma $\ref{lem:top_induction_base}$. Hence, a straightforward inductive argument shows that for each $i=1, \ldots, k$ we have
\[\mu(z_i,m)=\sigma_i(\mu(z,m)),\]
which finishes the proof of the lemma.

\end{proof}

\begin{proposition}\label{prop:key_words_triv}
Let $\rho =w_k \ldots w_1 \in \FF$ be any element of length $k$ in the free group. Consider any $z \in \Omega_n$ with $d(z) \geq 3k+6$. Then the following two conditions are equivalent.

\begin{enumerate}
    \item $\pi(\rho)=id_F$ in Thompson's group $F$.
    \item $\rho(X_n,Y_n)(z)=z$.
\end{enumerate} 
\end{proposition}

\begin{proof}

First, we prove the implication $(1) \implies (2)$. Suppose that $\pi(\rho)=id_F$ in Thompson's group $F$. Let $z=(L,R)$ and set 
\[z_k=(L_k,R_k)=\rho(X_n,Y_n)(z).\] We are going to show that $L_k=L$ and $R_k =R$.

Notice that for any $m \in [n]$ with $l_z(m) \geq k+3$, we already have
\[L(m)=L_k(m) \text{ and } R(m)=R_k(m),\]
by Lemma~\ref{lem:no_outchanges}. Moreover, since $\sigma: F \rightarrow \Sym(W)$ is a homomorphism, equality $\pi(\rho)=id_F$ implies that $\sigma_{\pi(\rho)}=id$ in $\Sym(W)$. Then, by Lemma~\ref{lem:addr_at_the_top}, for every $m \in B(z,k+4)$ we have 
\begin{equation}\label{eq:all_fixed}
\mu(z_k,m)=\mu(z,m).    
\end{equation}
Now, the conditions $d(z_k) \geq 2k+6$ and $d(z) \geq 3k+6$ together with \eqref{eq:all_fixed} imply the following equalities for every $m \in B(z,k+3)$:
\[\mu\bigl(z_k,L_k(m)\bigr)=\mu(z_k,m)0=\mu(z,m)0=\mu\bigl(z,L(m)\bigr)=\mu\bigl(z_k,L(m)\bigr),\]
Indeed, the equalities 
\[\mu(z_k,m)0=\mu(z,m)0 \, \text{ and } \, \mu\bigl(z,L(m)\bigr)=\mu\bigl(z_k,L(m)\bigr)\] follow from \eqref{eq:all_fixed} together with the observation that $m, \,L(m) \in B(z,k+4)$. And the equalities 
\[\mu\bigl(z_k,L_k(m)\bigr)=\mu(z_k,m)0 \, \text{ and }\,\mu(z,m)0=\mu\bigl(z,L(m)\bigr)\]
follow from the uniqueness of a minimal directed path from $0$ to any given vertex in $B(z_k,k+4)$ and in $B(z,k+4)$.

Since $d(z_k) \geq 2k+6$, for every $m \in B(z,k+3)$ the equality
\[\mu\bigl(z_k,L_k(m)\bigr) = \mu\bigl(z_k,L(m)\bigr)\]
together with uniqueness of addresses in $B(z_k,k+4)$ implies that
\[L(m)=L_k(m).\]
Finally, since this equality holds for every $m \in B(z,k+3)$, we have $L=L_k$ on the whole set $[n]$.

 The equality $R(m)=R_k(m)$ for $m \in B(z,k+3)$ is proved similarly, and it also immediately leads to $R=R_k$ on $[n]$. Thus, we indeed have  $z=z_k$. This finishes the proof of the implication $(1) \implies (2)$.

Next, we prove the implication $(2) \implies (1)$. Let us show that the negation of $(1)$ implies the negation of $ (2)$. Assume that $\pi(\rho)\ne id_F$ and consider the prefix-replacement representation of $\pi(\rho)$.  Suppose that in this representation, $\pi(\rho)$ replaces the prefix code $(\alpha_1 ,\ldots, \alpha_l)$ by the code $(\beta_1, \ldots, \beta_l)$. Since $\rho$ has length $k$, no prefix in this representation might have the length greater than $k+2$. Moreover, as this representation is non-trivial, there exist two distinct words $\alpha, \beta \in W$ of length at most $k+2$ such that $P_{\pi(\rho)}(\alpha)=\sigma_{\pi(\rho)}(\alpha)=\beta$.  Thus, the number $m$ with $\mu(z,m)=\alpha$ changes its address in $z_k$. However, since $d(z), d(z_k) \geq 2k+6,$ the addresses of numbers determine the restrictions of $\Gamma(z_k)$ and $\Gamma(z)$ to $B(z_k, k+3)$ and $B(z,k+3)$, respectively, completely. Therefore, $z_k$ and $z$ will differ already at some vertex of the ball $B(z,k+3)$.

\end{proof}

\begin{remark}
The proof of the implication $ \pi(\rho)\ne id_F \implies\rho(X_n,Y_n)(z) \ne z$ relies mostly on the fact that $\sigma$ extends the partial action $P$ of $F$ on $W$. However, the fact that $\sigma$ is indeed an action is crucial for the implication $ \pi(\rho) = id_F \implies\rho(X_n,Y_n)(z) = z$.
\end{remark}

\FloatBarrier

\section{Computation of $|G_n(r)|$ and the proof of Theorem~\ref{thm:main}}\label{sec:finishing_the_proof}

\begin{lemma}\label{lem:regular_set_prob}
For any $ r\geq 1$ and any $n\geq 2^{r+1}-1$ we have the equality
\begin{equation}\label{eq:exact-probability}
 \Prob_n(G_n(r))=\prod_{j=1}^{2^{r+1}-2}\left(1-\frac{j}{n}\right).
\end{equation}
In particular, 
\begin{equation}\label{eq:collision}
 \Prob_n(G_n(r)) \geq 1 -\frac{(2^{r}-1)(2^{r+1}-1)}{n}.
\end{equation}
\end{lemma}
\begin{proof}
Let $W_r$ be the set of all binary words of length at most $r$ equipped with the length-lexicographic order. It is easy to see that $|W_r|=2^{r+1}-1$.
Note that elements of $G_n(r)$ can be constructed by the following two-stage procedure. 

\begin{enumerate}
\item To define $s=(L,R) \in G_n(r)$, we first select an ordered list $A$ of $2^{r+1}-2$ distinct numbers from $[n] \setminus\{0\}$. Next, we define $s=(L,R)$ on the first $2^r-2$ elements of $A$ and $0$ as follows. Let $\mu:\{0\} \cup A \to W_r$ be the order-preserving bijection, so that for example, $\mu(0)=\emptyword$ and $\mu$ maps the largest $2^r$ elements of $A$ to the words of length $r$. Then, for any $a \in \{0\} \cup A$ if $\mu(a)=w \in W_r$ and $w0$ and $w1$ are still in $W_r$, define $L(a)=\mu^{-1}(w0)$ and $R(a)=\mu^{-1}(w1)$. This assignment guarantees that $d(s) \geq r$, and on $A$ the map $\mu$ is equal to the address map $\mu(s,\cdot)$ defined before. 

Note that the first stage can be performed in exactly
\[(n-1)(n-2) \ldots (n-2^{r+1}+2) \]
ways.

\item At the second stage, we need to define $s$ on the largest $2^r$ elements of $A$ and on every element of the set $[n]\setminus (A \cup \{0\})$. The rows of $T(L,R)$ indexed by all these numbers are empty, so we can fill them with the elements of $[n]$ arbitrarily. Therefore, the second stage can be performed in $n^{n-2^r+1}\times n^{n-2^r+1}=n^{2n-2^{r+1}+2}$ ways.
\end{enumerate}

As a result, we showed the equality
\[|G_n(r)|=(n-1)(n-2) \ldots (n-2^{r+1}+2)n^{2n-2^{r+1}+2}.\]
After we divide both sides by $n^{2n}$, we obtain the formula
\eqref{eq:exact-probability}.
Finally, recall that for any $k \in \mathbb{N}$ and any numbers $t_1, \ldots, t_k$ such that $0\leq t_j\leq 1, \,  j =1 \ldots k$, we have
\[1-\prod_{j=1}^k(1-t_j)\leq\sum_{j=1}^k t_j.\] 
Combining it with the formula \eqref{eq:exact-probability} we obtain the inequality
\[
 1- \Prob_n(G_n(r))\leq\sum_{j=1}^{2^{r+1}-2}\frac{j}{n}
 =\frac{(2^{r}-1)(2^{r+1}-1)}{n},
\]
which is equivalent to \eqref{eq:collision}.
\end{proof}

Now we are ready to prove a quantitative version of Theorem~\ref{thm:main} and show that the maps $\varphi_n: F \to \Sym(\Omega_n)$ that naively extend the map $x \mapsto X_n$ and $y \mapsto Y_n$ to $F$ indeed define a sequence of sofic approximations to $F$.

\begin{definition}
For any $g \in F$ let $\omega_g \in \FF$ denote the smallest element in the length-lexicographic order in the free group that projects to $g$, i.e. we have $\pi(\omega_g)=g$. 
\end{definition}

\begin{theorem}\label{thm:F_sofic_quantitative}
For any $\epsilon \in (0,1)$ and any $ R \in \mathbb{N}$, set 
\[m(R, \epsilon)=\dfrac{(2^{9R+6}-1)(2^{9R+7}-1)}{\epsilon}.\]
For any $n \in \mathbb{N}$ consider the map $\varphi_n: F \to \Sym(\Omega_n)$ defined by $g \mapsto \omega_g(X_n,Y_n)$. Then, for any $n \geq m(R,\epsilon)$, the map $\varphi_n$ is a $(B(R),\epsilon)$-sofic approximation for $F$, where $B(R)$ denotes the ball of radius $R$ in the Cayley graph of $F$ with respect to its standard generating set. Moreover, the Schreier graph of the action of $\langle X_n, Y_n \rangle$ on $\Omega_n$, with the marks $X_n^{\pm1}$ and $Y_n^{\pm1}$ replaced by $x^{\pm1}$ and $y^{\pm1}$, respectively, is an $(R,\epsilon)$ marked approximation to $F$.

\end{theorem}

\begin{proof}
 Lemma~\ref{lem:regular_set_prob}, applied to $r=9R+6$, implies that for any $n \geq m(R, \epsilon)$  we have
\[\Prob_n(G_n(9R+6)) \geq 1 - \epsilon.\]
Moreover, Proposition~\ref{prop:key_words_triv}, applied to $k=3R$, implies that for every word $\rho \in \FF$ of length at most $3R$, and every configuration $z \in G_n(9R+6)$, we have \[\pi(\rho)= id_F \iff \rho(X_n,Y_n)(z)=z.\]

In particular, for $g \in B(R)$ with $g \neq id_F$, we have $\phi_n(g)(z)=\omega_g(X_n,Y_n)(z) \ne z$, so 
\[\dH(\varphi_n(g), \id) \geq \Prob_n(G_n(9R+6)) \geq 1 - \epsilon.\]
Moreover, for any  $g,h \in B(R)$ such that $gh \in B(R)$,  the word $\rho_{g,h}=\omega_g\omega_h\omega_{gh}^{-1}$ has length at most $3R$ and evaluates to the trivial element in $F$. Thus, for any $z \in G_n(9R+6)$, and we have
\[\rho_{g,h}(X_n,Y_n)(z)=z.\]
But this implies the inequality
\[\dH\bigl(\varphi_n(g)\varphi_n(h),\varphi_n(gh)\bigr)=\dH\bigl(\varphi_n(g)\varphi_n(h)\varphi_n(gh)^{-1}, \id\bigr) \leq  1 - \Prob_n(G_n(9R+6)) \leq \epsilon.\]

Therefore, $\varphi_n$ indeed satisfies the definition of a $(B(R), \epsilon)$-sofic approximation to $F$.

Finally, the last assertion in the statement of the theorem follows from the fact that for a fixed set of marks, the isomorphism class of a rooted labeled ball with of radius $R$ is determined by the set of loops of length at most $2R+1$ in this ball. Then, by Proposition~\ref{prop:key_words_triv}, the rooted labeled ball of radius $R$ centered at a point from $G_n(9R+6)$ in the Schreier graph of the action of $\langle X_n, Y_n \rangle$ on $\Omega_n$, with its edge labels updated to $x^{\pm1}$ and $y^{\pm1}$, is always isomorphic to the ball of radius $R$ in the standard Cayley graph of $F$.
\end{proof}

\FloatBarrier
\section{Proof of Theorem~\ref{thm:V-intro}}\label{sec:final_proof_for_V}
 In this section we explain how to adapt the arguments from Sections~\ref{sec:fw_action}--\ref{sec:finishing_the_proof} to obtain the proof of soficity for Thompson's group $V$.
In fact, we will prove an even stronger quantitative version of Theorem~\ref{thm:V-intro}.

\begin{theorem}\label{thm:V-quantitative-sofic}
Thompson's group $V$ is sofic. Moreover, there exist a generating set $S=\{X,Y\}$ of $V$ and a constant $C>0$ such that the following statement holds. For any $\epsilon \in (0,1)$, any radius $R>0$, and any $n$ with 
\[n \geq \dfrac{(2^{18CR+1}-1)(2^{18CR+2}-1)}{\epsilon},\] there exists a map $\varphi_n: V \to \Sym(\Omega_n)$ that defines a $(B(R),\epsilon)$-sofic approximation for $V$, where $B(R)$ is the ball of radius $R$ centered at $id_V$ in the Cayley graph of $V$ with respect to $S$.
Moreover, there exists an $(R,\epsilon)$ marked approximation to $V$ with the vertex set $\Omega_n$.
\end{theorem}
 
\begin{proof}[Proof of Theorem~\ref{thm:V-quantitative-sofic}.]
Let us denote by $P$ the partial action of $V$ on the set of all binary words $W$ by prefix replacement transformations. 
First, we assume that there exist two elements $X$ and $Y$ of Thompson's group $V$ that satisfy the following two properties.

\begin{enumerate}
    \item The subgroup $V_1=\langle X, Y\rangle$ of $V$ is isomorphic to $V$.
    \item There exists a faithful action $\sigma$ of $V_1$ on the set of all finite binary words $W$ that fixes $\emptyword$ and such that for every $g \in V_1$, $\sigma_g$ extends $P_g$.
\end{enumerate}

We claim that this assumption implies the statement of our theorem. Indeed, it is easy to see that it suffices to construct suitable sofic approximations for $V_1$.
Consider the representations of $X$ and $Y$ by prefix-replacement transformations in $V$, and let $C>0$ be the largest possible length of a prefix in the representations of $X$ and $Y$. Note that the choice of $C$ and the fact that $\sigma_t$ extends $P_t$ for $t \in \{X^{\pm1}, Y^{\pm1}\}$ imply that for any $w \in W$, we have 
\[\bigl||w|-|\sigma_t(w)|\bigr| < C.\]

Next, we construct sofic approximations for $V_1$ by permutations from $\Sym(\Omega_n)$ using the same method that worked for the  group $F$. 

First, we will define the maps $X_n$, $Y_n$, $X_n^{-1}$ and $Y_n^{-1} $ as partial bijections of $\Omega_n$. Pick an arbitrary $t \in \{X,Y,X^{-1}, Y^{-1}\}$ and define the corresponding partial map $t_n$ as follows\footnote{At this stage, we formally do not assume that the maps denoted by $t_n$ and $t_n^{-1}$ are inverses of each other, but at the end of the construction we will indeed have $t_n^{-1}=(t_n)^{-1}$}. Consider the prefix replacement representation of $t$ in $V$,  $t:(\alpha_1 , \ldots, \alpha_l) \mapsto (\beta_1 , \ldots, \beta_l)$.
First, let us define the \textbf{regular set} for $t_n$ in $\Omega_n$. Recall that for $z \in {\Omega_n}$ and $\mu \in W$, we denoted by $\theta_z(\mu)(0)$ the endpoint of the path started at $0$ and coded by $\mu$ in $\Gamma(z)$. To simplify the notation, for the remainder of the argument, we will introduce the map $\tau_z: W \to [n]$ defined by $\mu \mapsto\theta_z(\mu)(0)$. 
Let $I_t$ denote the set of all proper prefixes of the words $(\alpha_1 , \ldots, \alpha_l)$ in $W$. Then $z \in \Omega_n$ belongs to the regular set for $t_n$ if and only if the map $\mu \mapsto \tau_z(\mu)$ is an injective map from $I_t \cup \{\alpha_1 , \ldots, \alpha_l\}$ to $[n]$. Note that since $\sigma_t$ extends $P_t$ on $W$, it bijectively maps $I_t \cup \{\alpha_1 , \ldots, \alpha_l\}$ to the set of all prefixes of the words from $\{\beta_1 , \ldots, \beta_l\}$, and $\sigma_t(I_t)$ is the set of all proper prefixes of the words from $\{\beta_1 , \ldots, \beta_l\}$, so $ \sigma_t(I_t)=I_{t^{-1}}$.

For $z$ in the regular set for $t_n$, we will say that $m \in [n]$ is \textbf{an internal vertex} of $t_n$ in $z$, if $m \in \tau_z(I_t)$, and we will denote by $\mathcal{I}_{t_n}$ the set of all internal vertices of $t_n$ in $z$. Similarly, we say that $m$ is \textbf{a leaf} of $t_n$ in $z$ if $m \in \tau_z(\{\alpha_1 , \ldots, \alpha_l\})$, and the set of the leaves of $t_n$ will be denoted by $\mathcal{P}_{t_n}$. 

For $z$ in the regular set for $t_n$, we can define an inverse map $\tau_z^{-1}: \mathcal{I}_{t_n} \cup \mathcal{P}_{t_n} \to I_t \cup \{\alpha_1 , \ldots, \alpha_l\}$. Thus, the path coded by the word $\tau_z^{-1}(m)$ leads from $0$ to $m$ in $\Gamma(z)$, but in general, $\tau_z^{-1}(m)$ does not have to be the minimal in the length-lexicographic order word with this property.

Now, let $z=(L,R)$ be a point in the regular set for $t_n$. Then, $t_n$ will map $z=(L,R)$ to $z'=(L',R')$ defined as follows. 

For any $ m \notin \mathcal{I}_{t_n}$, we set $L'(m)=L(m)$ and $R'(m)=R(m)$. Then we define $L'$  and $R'$ on $\mathcal{I}_{t_n}$ in a way that ensures that $z'$ satisfies the following two conditions.

\begin{enumerate}
    \item The configuration $z'$ is in the regular set for the map denoted by the symbol $t_n^{-1}$.
    
    \item For every $m \in \mathcal{I}_{t_n} \cup \mathcal{P}_{t_n}$, the path in $\Gamma(z')$ started at $0$ and coded by $\sigma_t(\tau_z^{-1}(m))$ ends at $m$.
\end{enumerate} 

The procedure for defining $L'(m)$ is explicitly described as follows. For any $m \in \mathcal{I}_{t_n}$, we first evaluate $\tau_z^{-1}(m) \in I_t$. Then, the word $\sigma_t(\tau_z^{-1}(m))$ is in $I_{t^{-1}}$, so  
\[\sigma_t(\tau_z^{-1}(m))0 \in I_{t^{-1}} \cup \{\beta_1 , \ldots, \beta_l\},\] and we have \[\sigma_t^{-1}(\sigma_t(\tau_z^{-1}(m))0) \in I_t \cup \{\alpha_1 , \ldots, \alpha_l\}\]
Then, we set $L'(m)$ to be equal to the endpoint $m'$ of a path in $\Gamma(z)$ started at $0$ and coded by the word $\sigma_t^{-1}(\sigma_t(\tau_z^{-1}(m))0)$. Note that we have  $m' \in \mathcal{I}_{t_n}\cup \mathcal{P}_{t_n}$, and $m'$ is the unique number in $\mathcal{I}_{t_n} \cup \mathcal{P}_{t_n}$ such that 
\begin{equation}\label{eq:left}
\sigma_t(\tau_z^{-1}(m'))=\sigma_t(\tau_z^{-1}(m))0
\end{equation}

Similarly, the image $R'(m)$ is chosen as the unique number $m'' \in \mathcal{I}_{t_n} \cup \mathcal{P}_{t_n}$ such that
\begin{equation}\label{eq:right}
 \sigma_t(\tau_z^{-1}(m''))=\sigma_t(\tau_z^{-1}(m))1.   
\end{equation}

It is not difficult to see that for any $m \in \mathcal{I}_{t_n} \cup \mathcal{P}_{t_n}$, by reading the word $\sigma_t(\tau_z^{-1}(m))$ backwards and applying the rules \eqref{eq:left} and \eqref{eq:right}, we indeed recover the vertices of a directed path from $0$ to $m$ in $\Gamma(z')$ which is coded by the word $\sigma_t(\tau_z^{-1}(m))$.
Since \[I_{t^{-1}} \cup \{\beta_1 , \ldots, \beta_l\}=\sigma_t(I_t \cup \{\alpha_1 , \ldots, \alpha_l\}),\]
we may conclude that all paths in $\Gamma(z')$ started at $0$ and coded by the words from $I_{t^{-1}} \cup \{\beta_1 , \ldots, \beta_l\}$ 
end at different points, so $z'$ is indeed in the regular set for $t_n^{-1}$. Moreover, we also have $t_n^{-1}(z')=z$, so the partial maps denoted by the symbols $t_n$ and $t_n^{-1}$ are indeed inverses of each other. Then, we arbitrarily define $X_n$ and $Y_n$ off their regular sets to obtain permutations of $\Omega_n$, and extend $X_n^{-1}$ and $Y_n^{-1}$  off their regular sets in a compatible way, so that we have $(X_n)^{-1}=X_n^{-1}$ and $(Y_n)^{-1}=Y_n^{-1}$ in $\Sym(\Omega_n)$.

With the construction of $X_n^{\pm1}$ and $Y_n^{\pm1}$ complete, we observe that the following version of Remark~\ref{rem:Xn_and_sigma_x} and Remark~\ref{rem:Yn_and_sigma_y} holds in our setting.

\begin{claim}
Consider any $t \in \{X,Y,X^{-1}, Y^{-1}\}$ and any $z \in G_n(2C+1)$. Then $z$ belongs to the regular set of $t_n$. Moreover, if we set $z'=t_n(z)$, then for any $m \in \mathcal{I}_{t_n} \cup\mathcal{P}_{t_n}$, we have 
\[\mu(z',m)=\sigma_t(\mu(z,m)).\]
\end{claim}

\begin{proof}
Consider the prefix replacement representation of $t$ in $V$,  $t:(\alpha_1 , \ldots, \alpha_l) \mapsto (\beta_1 , \ldots, \beta_l)$. By the definition of $C$, no prefix in these lists may be longer than $C$. However, the condition $z \in G_n(2C+1)$ implies  that all paths of length at most $2C+1$ started at $0$ in $\Gamma(z)$ have distinct endpoints, so 
the map $\tau_z: I_t \cup\{\alpha_1, \ldots, \alpha_l\}\to[n]$ is injective.

Moreover, in this case, 
for every $m \in B(z, 2C+1)$ there is exactly one directed path from $0$ to $m$ of length at most $2C+1$, and this path is coded by the address $\mu(z,m)$. Hence, for any $m \in \mathcal{I}_{t_n} \cup\mathcal{P}_{t_n}$ we have \[\tau_z^{-1}(m)=\mu(z,m).\]
In the construction of $t_n$, we already showed that for any $m \in \mathcal{I}_{t_n} \cup\mathcal{P}_{t_n}$
the path in $\Gamma(z')$ started at $0$ and coded by $\sigma_t(\tau_z^{-1}(m))$ ends at $m$. Since we have $\tau_z^{-1}(m)=\mu(z,m)$, we see that there is a path in $\Gamma(z')$ started at $0$ and coded by $\sigma_t(\mu(z,m))$ which ends at $m$. 
Moreover, since $z'$ belongs to the regular set of $t_n^{-1}$, we conclude that it is the only path from $0$ to $m$ in $\Gamma(z')$ that visits only the numbers from $\mathcal{I}_{t_n}$ before reaching $m$.

Therefore, to conclude that $\sigma_t(\mu(z,m))$ is indeed the address of $m$ in $z'$, it suffices to show that any path $\gamma$ from $0$ to $m$ that visits some vertex $h \notin \mathcal{I}_{t_n}$, before it reaches $m$, will be longer than $|\sigma_t(\mu(z,m))|$. Indeed, by the construction of $z'$, the first vertex outside of $\mathcal{I}_{t_n} $ visited by $\gamma$ will be some leaf $h$ of $t_n$ in $z$. Note that we have $(L'(m),R'(m))=(L(m),R(m))$ for any $m \notin \mathcal{I}_{t_n}$. Thus, the inequality $|\mu(z,h)| \leq C$ together with the assumption $z \in G_n(2C+1)$ imply that any path started at $h$ in $\Gamma(z')$ will visit only the vertices outside of $\mathcal{I}_{t_n} \cup \mathcal{P}_{t_n}$ for the first $C+1$ steps. Thus, it will be longer than the path coded by $\sigma_t(\mu(z,m))$ that has the length at most $C$.
Therefore, we indeed have $\mu(z',m)=\sigma_t(\mu(z,m))$ for each $m \in \mathcal{I}_{t_n} \cup \mathcal{P}_{t_n}$.
\end{proof}

This claim allows us to keep track of the addresses of internal vertices and leaves of $t\in \{X_n^{\pm1}, Y_n^{\pm1}\}$ if we apply $t$ to a configuration $z$ from $G_n(r)$ with $r \geq 2C+1$.
The rest of the proof for $V_1$ follows the arguments from Section~\ref{sec:maps_and_approximations} and Section~\ref{sec:finishing_the_proof} with only minor technical updates in the statements of Lemmas \ref{lem:xy_on_good_paths} -- \ref{lem:addr_at_the_top} and Proposition~\ref{prop:key_words_triv}.
We list these changes below.
\begin{enumerate}
    \item Assumption $z \in G_n(4)$ in Lemma~\ref{lem:xy_on_good_paths} is replaced by $z \in G_n(2C+1)$.
    \item The assumption $d(z) \geq 5$ in Lemma~\ref{lem:regularity persists} is replaced by $ d(z) \geq 3C+1$, and the inequality $d(z') \geq d(z)-1 $ is replaced by $d(z') \geq d(z)-C$.
    \item The assumption $d(z) \geq 5$ in Lemma~\ref{lem:length_control} is replaced by $ d(z) \geq 3C+1$, and the inequality $|l_{z'}(m) -l_z(m)| \leq 1 $ is replaced by $|l_{z'}(m) -l_z(m)| \leq C $.
    \item The assumption $d(z) \geq 5$ in Lemma~\ref{lem:top_induction_base} is replaced by $ d(z) \geq 5C+1$, and the assumption $l_{z}(m) \leq d(z)-2$ is replaced by $l_{z}(m) \leq d(z)-2C$.
    \item The assumption $d(z) \geq k+6$ in Lemma~\ref{lem:no_outchanges} is replaced by $d(z) \geq C(k+3)+1$. The assumption $l_z(m) \geq k+3$ is replaced by $l_z(m) \geq C(k+1)$.
    \item The assumption $d(z) \geq 3k+6$ in Lemma~\ref{lem:addr_at_the_top} is replaced by $d(z) \geq 6Ck+1$ and the assumption $l_z(m) \leq k+4 $ is replaced by the assumption $l_z(m) \leq C(k+1)$. 
    \item The assumption $d(z) \geq 3k+6$ in Proposition~\ref{prop:key_words_triv} is replaced by $d(z) \geq 6Ck+1$.
\end{enumerate}

In Section 5, Lemma~\ref{lem:regular_set_prob} does not need any update, and Theorem~\ref{thm:F_sofic_quantitative} still holds for $V_1$ with \[m(R, \epsilon)=\dfrac{(2^{18CR+1}-1)(2^{18CR+2}-1)}{\epsilon}.\]

Finally, a sofic approximation for $V$ is obtained by taking a composition of an isomorphism between $V_1$ and $V$ and the approximation for $V_1$.

It remains to find suitable elements $X$ and $Y$ in $V$ and the action $\sigma$ of $V_1$ on $W$. 

We claim that $V_1$ can be chosen as a rigid stabilizer of $11\mathcal{C}$ for the standard action of $V$ on $\mathcal{C}$.  
Indeed, let us set $V_1=\mathrm{Rist}_V(11\mathcal{C})$. It is clear that $V_1$ is isomorphic to $V$, and since $V$ is generated by two elements, see for example \cite[Theorem 1.3]{BQ}, we can find $X,Y \in V_1$ that generate $V_1$. Thus, it remains to check that $V_1$ admits a faithful action on $W$ that fixes $\emptyword$ and has the following property. For each $g \in V_1$, we have $\sigma_g(w)=P_g(w)$ for every $w \in W$ where $P_g$ is well-defined.

Let $D \subset \mathcal{C}$ be the set of all binary sequences $\omega \in \mathcal{C}$ that eventually become $0$. Consider an injective map $j:W \rightarrow D$ defined by 
\[j(w)=w10^{\infty}.\]
It is easy to see that $j$ is a bijection between $W$ and $D\setminus\{0^{\infty}\}$ and $j(\emptyword)=10^{\infty}$ is a fixed point of $V_1$. Moreover, since $V_1$ fixes $0^{\infty}$, $V_1$ acts on $D\setminus\{0^{\infty}\}$, and this action is faithful, because $D\setminus\{0^{\infty}\}$ is dense in $\mathcal{C}$. We claim that $g \mapsto \sigma_g=j^{-1}gj$ produces the desired action of $V_1$ on $W$. 
It remains to show that for any word $w \in W$ such that $P_g$ is defined, we have 
\[jP_g(w)=gj(w).\]
It is clearly true for the identity element, so we may assume that $g\ne id$. Then, since $P_g(w)$ is well-defined, one of the words $0$, $10$, and $11$ is a prefix of $w$. If $w$ has prefix $0$ or $10$, then $w$ is fixed by $P_g$ and $j(w)$ is fixed by $g$, so we get:
\[jP_g(w)=j(w)=gj(w).\]

Otherwise, $w$ has the form $w=11\alpha w'$ where $\alpha, \,w' \in W$, and $g$ replaces the prefix $11\alpha$ by a prefix of the form $11\beta$ for some $\beta$. But then we have 

\[jP_g(w)=j(11\beta w')=11\beta w'10^{\infty}=g(11\alpha w'10^{\infty})=g(j(w)).\]

Therefore, $jP_g(w)=g(j(w))$ in this case as well, and the map $g \mapsto \sigma_g$ defines the desired action of $V_1$ on $W$. This finishes the proof of the theorem.

\end{proof}

\begin{remark}
It is not difficult to see that the proof of Theorem~\ref{thm:V-quantitative-sofic} does not depend on the fact that $V$ is generated by $2$ elements or on the fact that we are working with the words over the alphabet $\{0,1\}$. In particular, it can be generalized to the Higman-Thompson groups $V_n$ for all $n \geq 2$. Moreover, it will also work for the stabilizer of $\emptyword$ in the group of color-preserving quasi-automorphisms $QAut(\mathcal{T}_{2,c})$. However, soficity of all of these groups immediately follows from soficity of $V$, as we explain in Corollary~\ref{cor:sofic_bc_in_V} below.
\end{remark}

\begin{corollary}\label{cor:sofic_bc_in_V}
The following groups are sofic.
\begin{enumerate}
    \item The Higman-Thompson group $V_n$, for $n \geq 2$.
    \item The group $QAut(\mathcal{T}_{2,c})$.
    \item The groups of quasi-automorphisms $QV_{n,r,p}$ for $n \geq 2, r \geq 1, p \geq 0$.
    \item The generalized Houghton groups $H_{r,p}$, $r \geq 1$, $p \geq 0$.
    
\end{enumerate}
 
\end{corollary}
\begin{proof}
The Higman-Thompson groups $V_n$ embed into $V$ by the work of Higman, see \cite[Section 7]{H74}.
The group $QAut(\mathcal{T}_{2,c})$ embeds into $V$  by \cite[Theorem $4$]{BMN}. The groups $QV_{n,r,p}$ and $H_{r,p}$ embed into $V$ by \cite[Corollary $1.3$]{G19}.
\end{proof}

\end{document}